\documentclass{amsart}
\usepackage{amsmath,amssymb,graphicx,latexsym}
\usepackage{hyperref}
\usepackage{tikz}

\theoremstyle{plain}
\newtheorem{theorem}{Theorem}
\newtheorem{conjecture}[theorem]{Conjecture}
\newtheorem{corollary}[theorem]{Corollary}
\newtheorem{lemma}[theorem]{Lemma}
\newtheorem{proposition}[theorem]{Proposition}
\newtheorem*{regularitytheorem*}{Theorem~\ref{thm:regularity}}

\theoremstyle{definition}
\newtheorem{definition}[theorem]{Definition}
\newtheorem{example}[theorem]{Example}
\newtheorem{notation}[theorem]{Notation}
\newtheorem{remark}[theorem]{Remark}
\newtheorem*{notation*}{Notation}

\DeclareMathOperator{\Alphabet}{Alph}

\newcommand{\N}{\mathbb{N}}
\newcommand{\Q}{\mathbb{Q}}
\newcommand{\Z}{\mathbb{Z}}
\newcommand{\word}{\mathbf{w}}
\newcommand{\p}{\mathbf{p}}
\newcommand{\s}{\mathbf{s}}
\newcommand{\z}{\mathbf{z}}

\newcommand{\nequiv}{\mathrel{\not\equiv}}

\newcommand{\ceil}[1]{\left\lceil #1 \right\rceil}
\newcommand{\floor}[1]{\left\lfloor #1 \right\rfloor}
\newcommand{\size}[1]{\lvert #1 \rvert}

\newcommand{\citeseq}[1]{\cite[\href{http://oeis.org/#1}{#1}]{OEIS}}
\newcommand{\seq}[1]{\href{http://oeis.org/#1}{#1}}

\begin{document}

\title{Avoiding $5/4$-powers on the alphabet of nonnegative integers}\thanks{The second-named author is supported by a Francqui Foundation Fellowship of the Belgian American Educational Foundation.}
\author{Eric Rowland}
\author{Manon Stipulanti}
\address{
	Department of Mathematics \\
	Hofstra University \\
	Hempstead, NY \\
	USA
}
\date{May 6, 2020}

\begin{abstract}
We identify the structure of the lexicographically least word avoiding $5/4$-powers on the alphabet of nonnegative integers.
Specifically, we show that this word has the form $\p \, \tau ( \varphi(\z) \varphi^2(\z) \cdots)$ where $\p,\z$ are finite words, $\varphi$ is a $6$-uniform morphism, and $\tau$ is a coding.
This description yields a recurrence for the $i$th letter, which we use to prove
that the sequence of letters is $6$-regular with rank $188$.
More generally, we prove $k$-regularity for a sequence satisfying a recurrence of the same type.
\end{abstract}

\maketitle

\section{Introduction}\label{sec:Introduction}

Avoidance of patterns is a major area of study in combinatorics on words~\cite{Berstel--Perrin}, which finds its origins in the work of Thue~\cite{Berstel,Thue-1906,Thue-1912}.
In particular, lexicographically least words avoiding some patterns have gained interest over the years.
Often, words of interest that avoid a pattern can be described by a morphism.
A \emph{morphism} on an alphabet $\Sigma$ is a map $\mu \colon \Sigma \to \Sigma^*$.
(Here $\Sigma^*$ denotes the set of finite words on $\Sigma$.)
A morphism extends naturally to finite and infinite words by concatenation.
We say that a morphism $\mu$ on $\Sigma$ is \emph{$k$-uniform} if $\lvert \mu(c) \rvert = k$ for all $c \in \Sigma$.
A $1$-uniform morphism is also called a \emph{coding}.
If there exists a letter $c \in \Sigma$ such that $\mu(c)$ starts with $c$, then iterating $\mu$ on $c$ gives a word $\mu^{\omega}(c)$, which is a fixed point of $\mu$ beginning with $c$.
In this paper we index letters in finite and infinite words starting with $0$.

An \emph{overlap} is a word of the form $cxcxc$ where $c$ is a letter.
On a binary alphabet, the lexicographically least overlap-free word is $001001\varphi^\omega(1)$, where $\varphi(0) = 01, \varphi(1) = 10$ is the morphism generating the Thue--Morse word $\varphi^\omega(0)$~\cite{Allouche--Currie--Shallit}.

In the context of combinatorics on words, fractional powers were first studied by Dejean~\cite{Dejean}.
Such a power is a partial repetition, defined as follows.

\begin{definition}
Let $a$ and $b$ be relatively prime positive integers.
If $v = v_0 v_1 \cdots v_{\ell-1}$ is a nonempty word whose length $\ell$ is divisible by $b$, the \emph{$a/b$-power} of $v$ is the word
\[
v^{a/b} := v^{\lfloor a/b \rfloor} v_0 v_1 \cdots v_{\ell \cdot \{a/b\}-1} ,
\]
where $\{a/b\} = a/b - \lfloor a/b \rfloor$ is the fractional part of $a/b$.
\end{definition}

Note that $\lvert v^{a/b} \rvert = \frac{a}{b} \lvert v \rvert$.
If $a/b > 1$, then a word $w$ is an $a/b$-power if and only if $w$ can be written $v^eu$ where $e$ is a positive integer, $u$ is a prefix of $v$, and $\frac{\lvert w \rvert}{\lvert v \rvert} = \frac{a}{b}$.

\begin{example}
The $5/4$-power of the word $0111$ is $(0111)^{5/4} = 01110$.
\end{example}

In general, a $5/4$-power is a word of the form $(xy)^{5/4} = xyx$, where $\size{xy} = 4 \ell$ and $\size{xyx} = 5\ell$ for some $\ell \ge 1$.
It follows that $\size{x} = \ell$ and $\size{y} = 3\ell$.

Elsewhere in the literature, researchers have been interested in words with no $\alpha$-power factors for all $\alpha \ge a/b$.
In this paper, we consider a slightly different notion, and we say that a word is \emph{$a/b$-power-free} if none of its factors is an (exact) $a/b$-power.

\begin{notation*}
Let $a$ and $b$ be relatively prime positive integers such that $a/b > 1$.
Define $\word_{a/b}$ to be the lexicographically least infinite word on $\Z_{\ge 0}$ avoiding $a/b$-powers.
\end{notation*}

Guay-Paquet and Shallit~\cite{Guay--Shallit} started the study of lexicographically least power-free words on the alphabet of nonnegative integers.
They identified the structure of $\word_a$ for each integer $a \geq 2$.
In particular, the lexicographically least $2$-power-free word~\citeseq{A007814}
\[
	\word_2 = 01020103010201040102010301020105 \cdots
\]
is the fixed point of the $2$-uniform morphism $\varphi$ on the alphabet of nonnegative integers defined by $\varphi(n) = 0 (n + 1)$ for all $n\ge 0$.
Additionally they identified the structure of the lexicographically least overlap-free word.
The first-named author and Shallit~\cite{Rowland--Shallit} studied the structure of the lexicographically least $3/2$-power-free word~\citeseq{A269518}
\[
	\word_{3/2} = 0011021001120011031001130011021001140011031 \cdots,
\]
which is the image under a coding of a fixed point of a $6$-uniform morphism.
Let $\Sigma_2$ be the infinite alphabet $\{n_j \colon n\in \Z, 0\le j \le 1\}$ with 2 types of letters.
For example, $0_0$ and $0_1$ are the 2 different letters of the form $0_j$.
Let $\varphi \colon \Sigma_2^* \to \Sigma_2^*$ be the morphism defined by
\begin{align*}
	\varphi(n_0) &= 0_0 0_1 1_0 1_1 0_0 (n + 2)_1 \\
	\varphi(n_1) &= 1_0 0_1 0_0 1_1 1_0 (n + 2)_1
\end{align*}
for all $n\in \Z$, where the subscript $j$ determines the first five letters of $\varphi(n_j)$.
Let $\tau$ be the coding defined by $\tau(n_j)=n$ for all $n_j \in \Sigma_2$.
Then $\word_{3/2}=\tau(\varphi^\omega(0_0))$.
A prefix of this word appears on the left in Figure~\ref{fig:5over4 - 6}.
The letter $0$ is represented by white cells, $1$ by slightly darker cells, and so on.
The first five columns are periodic, and the sixth column satisfies $w(6 i + 5) = w(i) + 2$ for all $i\ge 0$ where $w(i)$ is the $i$th letter of $\word_{3/2}$.

Pudwell and Rowland~\cite{Pudwell--Rowland} undertook a large study of $\word_{a/b}$ for rational numbers in the range $1 < \frac{a}{b} < 2$, and identified many of these words as images under codings of fixed points of morphisms.
The number $\frac{a}{b}$ in this range with smallest $b$ for which the structure of $\word_{a/b}$ was not known is $\frac{5}{4}$.
In this paper, we give a morphic description for the lexicographically least $5/4$-power-free word~\citeseq{A277144}
\[
\word_{5/4} = 00001111020210100101121200001311\cdots.
\]
Let $w(i)$ be the $i$th letter of $\word_{5/4}$.
For the morphic description of $\word_{5/4}$, we need $8$ letters, $n_0, n_1, \dots, n_7$ for each integer $n \in \Z$.
The subscript $j$ of the letter $n_j$ will determine the first five letters of $\varphi(n_j)$, which correspond to the first five columns on the right in Figure~\ref{fig:5over4 - 6}.
The definition of $\varphi$ below implies that these columns are eventually periodic with period length $1$ or $4$.

\begin{figure}
	\includegraphics[width=.07\textwidth]{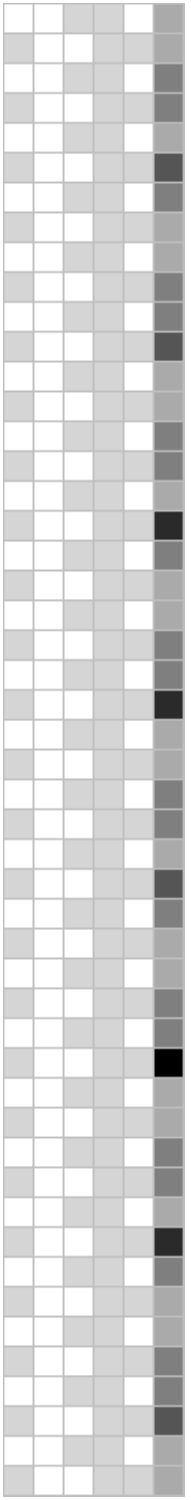}
	\qquad \qquad
	\includegraphics[width=.07\textwidth]{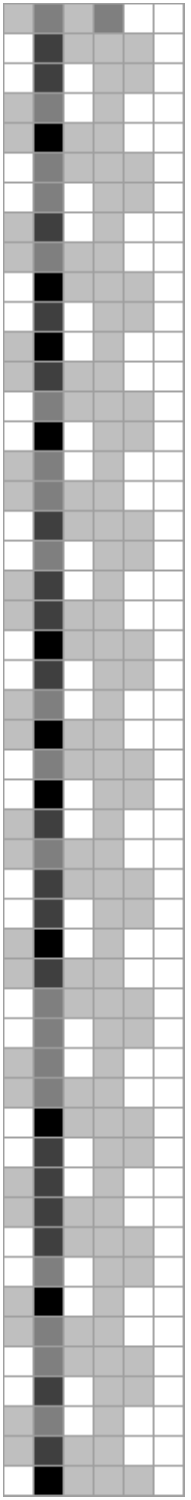}
	\qquad \qquad
	\includegraphics[width=.07\textwidth]{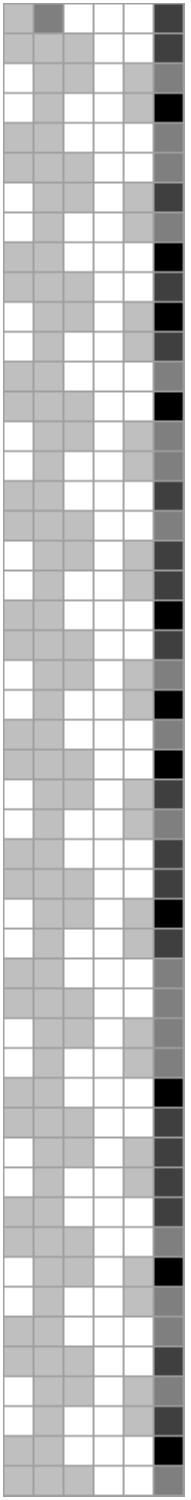}	
	\caption{Portions of $\word_{3/2}$ (left) and $\word_{5/4}$ (middle and right), partitioned into rows of width $6$.
	The word $\word_{3/2}$ is shown from the beginning.
	The word $\word_{5/4} = w(0) w(1) \cdots$ is shown beginning from $w(i)_{i \geq 6756}$ (middle) and $w(i)_{i \geq 6758}$ (right).
	In middle image, we have chopped off the first $6756/6=1126$ rows to show where five columns become periodic.
	The term $w(6759)$ (top row, second column on the right) is the last entry in $w(6 i + 3)_{i \geq 0}$ that is not $1$.}
	\label{fig:5over4 - 6}
\end{figure}

\begin{notation}\label{not:phi}
Let $\Sigma_8$ be the alphabet $\{n_j \colon n\in \Z, 0\le j \le 7\}$.
Let $\varphi$ be the $6$-uniform morphism defined on $\Sigma_8$ by
\begin{align*}
	\varphi(n_0) &= 0_0 1_1 0_2 0_3 1_4 (n + 3)_5 \\
	\varphi(n_1) &= 1_6 1_7 0_0 0_1 0_2 (n + 2)_3 \\
	\varphi(n_2) &= 1_4 1_5 1_6 0_7 0_0 (n + 3)_1 \\
	\varphi(n_3) &= 0_2 1_3 1_4 0_5 1_6 (n + 2)_7 \\
	\varphi(n_4) &= 0_0 1_1 0_2 0_3 1_4 (n + 1)_5 \\
	\varphi(n_5) &= 1_6 1_7 0_0 0_1 0_2 (n + 2)_3 \\
	\varphi(n_6) &= 1_4 1_5 1_6 0_7 0_0 (n + 1)_1 \\
	\varphi(n_7) &= 0_2 1_3 1_4 0_5 1_6 (n + 2)_7.
\end{align*}
We suggest keeping a copy of the definition of $\varphi$ handy, since we refer to it many times in the rest of the paper.
\end{notation}

The subscripts in each image $\varphi(n_j)$ increase by $1$ modulo $8$ from one letter to the next and also from the end of each image to the beginning of the next.
We also define the coding $\tau(n_j) = n$ for all $n_j\in \Sigma_8$.
In the rest of the paper, we think about the definitions of $\varphi$ and $\tau \circ \varphi$ as $8\times 6$ arrays of their letters.
In particular, we will refer to letters in images of $\varphi$ and $\tau \circ \varphi$ by their columns (first through sixth).

The following gives the structure of $\word_{5/4}$.

\begin{theorem}\label{thm:main}
There exist a word $\p$ on $\N = \{0, 1, \dots\}$ of length $6764$ and a word $\z$ on $\Sigma_8$ of length $20226$ such that $\word_{5/4} = \p \, \tau ( \varphi(\z) \varphi^2(\z) \cdots )$.
\end{theorem}

In particular, we can show that $\word_{5/4}$ is a morphic word.
Let $0'$ be a letter not in $\Sigma_8$, and define $\varphi(0') = 0' \z$.
Then $\varphi$ has the fixed point $\varphi^\omega(0') = 0' \z \varphi(z) \varphi^2(\z) \cdots$.
Closure properties~\cite[Theorems~7.6.1 and 7.6.3]{Allouche--Shallit-2003} for morphic words imply that chopping off the prefix $\tau(0'\z)$ of $\tau(\varphi^\omega(0'))$ and prepending $\p$ preserve the property of being a morphic word.
Therefore $\word_{5/4}$ is morphic.

Theorem~\ref{thm:main} also implies that five of the six columns of $\word_{5/4}$ are eventually periodic, and the last column satisfies the following recurrence.

\begin{corollary}\label{cor:5over4}
Let $w(i)$ be the $i$th letter of the word $\word_{5/4}$.
Then, for all $i\ge 0$,
\[
	w(6 i + 123061) =
	w(i + 5920) + \begin{cases}
		3	& \text{if $i \equiv 0, 2 \mod 8$} \\
		1	& \text{if $i \equiv 4, 6 \mod 8$} \\
		2	& \text{if $i \equiv 1 \mod 2$}.
	\end{cases}
\]
\end{corollary}

There are $20510$ transient rows before the self-similarity repeats, hence the value $123061=6\cdot 20510 +1$ in Corollary~\ref{cor:5over4}.
We show that the sequence $w(i)_{i\ge 0}$ is $6$-regular in the sense of Allouche and Shallit~\cite{Allouche--Shallit-1992}.
More generally, we prove the following.

\begin{theorem}\label{thm:regularity}
Let $k\ge 2$ and $\ell\ge 1$.
Let $d(i)_{i \geq 0}$ and $u(i)_{i \geq 0}$ be periodic integer sequences with period lengths $\ell$ and $k \ell$, respectively.
Let $r,s$ be nonnegative integers such that $r - s + k- 1 \ge 0$.
Let $w(i)_{i\ge 0}$ be an integer sequence such that, for all $0\le m \le k-1$ and all $i\ge 0$,
\[
w(ki+r+m)
=
\begin{cases}
	u(ki+m) & \text{if $0 \leq m \leq k-2$} \\
	w(i+s) + d(i) & \text{if $m = k-1$}.
\end{cases}
\]
Then $w(i)_{i\ge 0}$ is $k$-regular.
\end{theorem}

To prove Theorem~\ref{thm:main}, we must show that
\begin{enumerate}
\item $\p \, \tau ( \varphi(\z) \varphi^2(\z) \cdots )$ is $5/4$-power-free, and
\item $\p \, \tau ( \varphi(\z) \varphi^2(\z) \cdots )$ is lexicographically least (by showing that decreasing any letter introduces a $5/4$-power ending in that position).
\end{enumerate}

The word $\word_{5/4}$ is more complicated than previously studied words $\word_{a/b}$ in three major ways.
First, unlike all words $\word_{a/b}$ whose structures were previously known, the natural description of $\word_{5/4}$ is not as a morphic word, that is, as an image under a coding of a fixed point of a morphism.
This can be seen in Corollary~\ref{cor:5over4}.
Namely, when $w(i)_{i \geq 0}$ reappears as a modified subsequence of $\word_{5/4}$, it does not appear in its entirety; instead, only $w(i)_{i \geq 5920}$ appears.
In other words, there is a second kind of transient, represented by $5920 \neq 0$, which had not been observed before.
Second, the value of $d$ in the images $\varphi(n_j) = u \, (n + d)_i$ varies with $j$.
The sequence $3,2,3,2,1,2,1,2,\dots$ of $d$ values is periodic with period length $8$, hence the $8$ types of letters.
Third, the morphism $\varphi$ does not preserve the property of $5/4$-power-freeness, as we discuss in Section~\ref{sec:pre-5/4-power-freeness}.
These features make the proofs significantly more intricate.
We use \emph{Mathematica} to carry out several computations required in the proofs.
In particular, we explicitly use the length-$331040$ prefix of $\word_{5/4}$.
A notebook containing the computations is available from the websites\footnote{
\url{https://ericrowland.github.io/papers.html} and \url{https://sites.google.com/view/manonstipulanti/research}
} of the authors.

This paper is organized as follows.
Section~\ref{sec:first-prop} gives some useful preliminary properties of the words $\p$ and $\z$ from Theorem~\ref{thm:main}.
In Section~\ref{sec:length}, we show that $5/4$-powers in images under $\varphi$ have specific lengths.
As the morphism $\varphi$ does not preserve $5/4$-power-freeness, we introduce the concept of pre-$5/4$-power-freeness in Section~\ref{sec:pre-5/4-power-freeness} and we show that the word $\z\varphi(\z)\cdots$ is pre-$5/4$-power-free.
We prove Theorem~\ref{thm:main} in two steps.
First, in Section~\ref{sec:5/4-power-freeness}, we show that $\p\tau(\varphi(\z)\varphi^2(\z)\cdots)$ is $5/4$-power-free using the pre-$5/4$-power-freeness of $\z\varphi(\z)\cdots$.
Second, in Section~\ref{sec:lexicographic-least}, we show that $\p\tau(\varphi(\z)\varphi^2(\z)\cdots)$ is lexicographically least.
In Section~\ref{sec:6-regularity}, we study the regularity of words whose morphic structure is similar to that of $\word_{5/4}$, and we prove Theorem~\ref{thm:regularity}.
In particular, we prove that the sequence of letters in $\word_{5/4}$ is $6$-regular, and we establish that its rank is $188$.
We finish up with some open questions in Section~\ref{sec:open-questions}, including conjectural recurrences for $\word_{7/6}$ and several other words.

In the remainder of this section, we outline how the structure of $\word_{5/4}$ was discovered.

\subsection{Experimental discovery}

\begin{figure}
	\includegraphics[width=.5\textwidth]{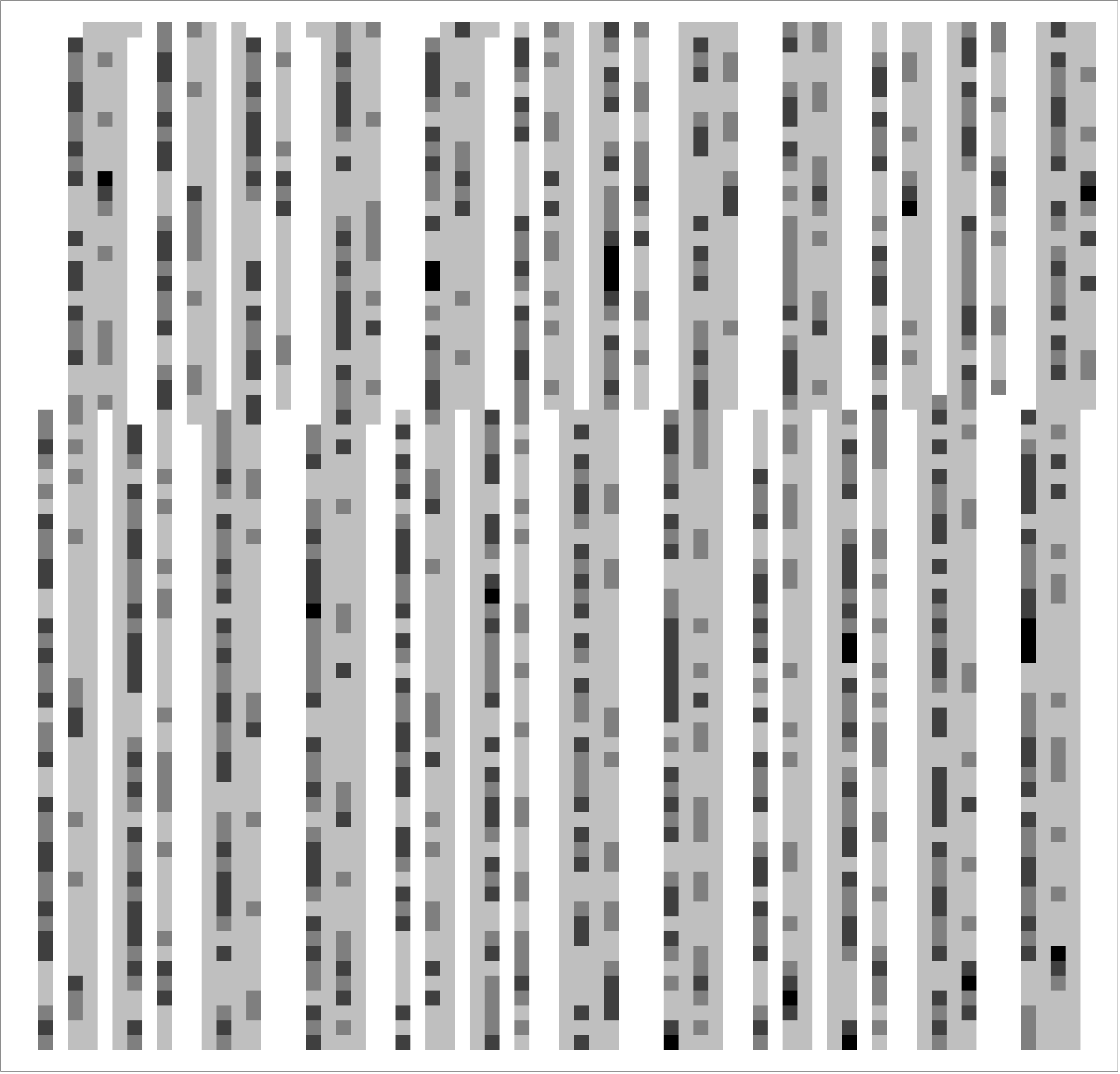}
	\caption{A prefix of $\word_{5/4}$, partitioned into rows of width $72$.}
	\label{fig:5over4}
\end{figure}

If the previous words studied in \cite{Rowland--Shallit} and \cite{Pudwell--Rowland} are any indication, the structure of $\word_{a/b}$ can be identified when the letters of $\word_{a/b}$ are partitioned into rows of width $k$ such that exactly one column is not eventually periodic.
We then look for the letters of $\word_{a/b}$ appearing self-similarly in this column.
For $\word_{5/4}$, the largest such $k$ appears to be $k = 72$.
Figure~\ref{fig:5over4} shows the partition of a prefix of $\word_{5/4}$ into rows of width $72$.

\begin{figure}
	\includegraphics[width=.45\textwidth]{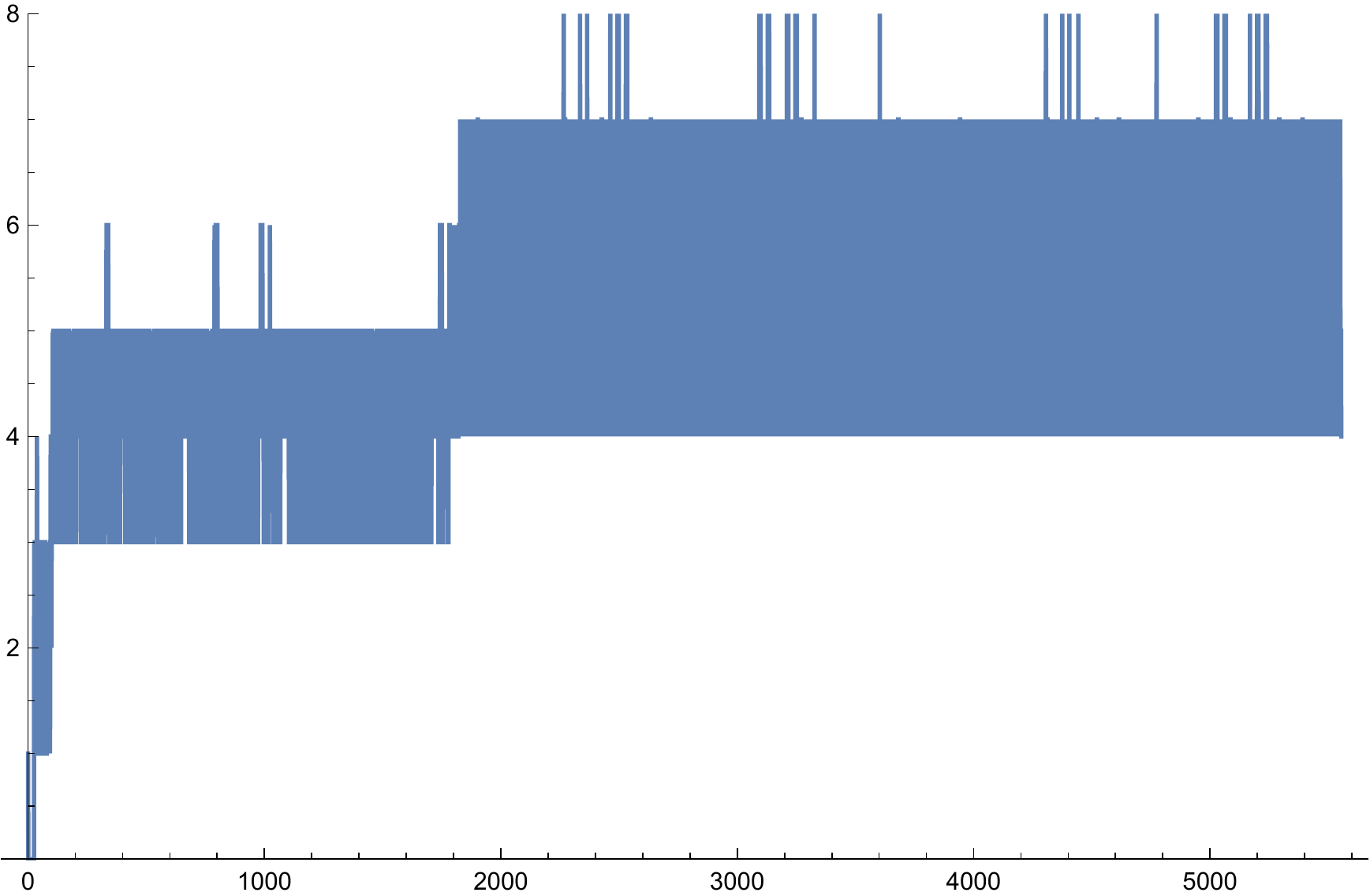}
	\qquad
	\includegraphics[width=.45\textwidth]{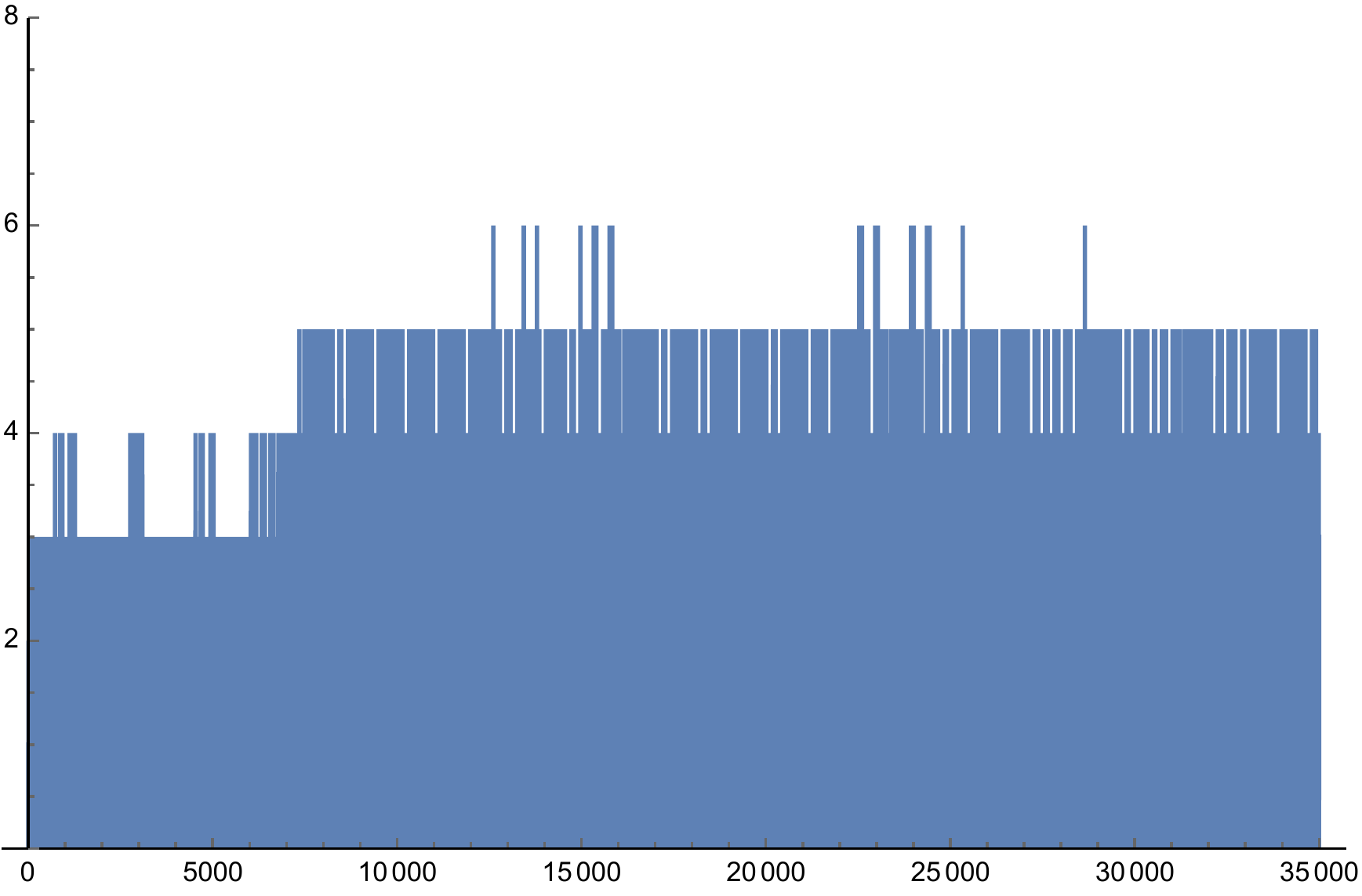}
	\caption{A plot of prefixes of the sequences $w(72 i + 31)_{i \geq 0}$ (left) and $\word_{5/4}=w(i)_{i \geq 0}$ (right).}
	\label{fig:5over4 plots}
\end{figure}

A longer prefix reveals that the sought nonperiodic column is $w(72 i + 31)_{i \geq 0}$.
Figure~\ref{fig:5over4 plots} plots the first several thousand terms of the sequences $w(72 i + 31)_{i \geq 0}$ and $w(i)_{i \geq 0}$.
The peaks in these plots suggest that the occurrences of the letter $8$ in $w(72 i + 31)_{i \geq 0}$ are related to the occurrences of the letter $6$ in $w(i)_{i \geq 0}$.
Specifically, the intervals between instances of $6$ in $w(i)_{i \geq 0}$ seem to be twelve times as long as the corresponding intervals between instances of $8$ in $w(72 i + 31)_{i \geq 0}$.
Lining up the peaks suggests
\begin{equation}\label{eqn:5over4 first conjecture}
	w(72 i + 163183) = w(12 i + 12607) + 2
\end{equation}
for all $i \geq 0$.
If possible, we would like a conjecture of the form
\begin{equation}\label{eqn:double transient general form}
	w(k i + r') = w(i + s) + d
\end{equation}
where the coefficient of $i$ on the right side is $1$, since such a recurrence would relate a subsequence of the letters in $\word_{5/4}$ to a suffix of $\word_{5/4}$.
Here $r'$ represents the ``usual'' transient seen in other words $\word_{a/b}$ (related to $r$ in Theorem~\ref{thm:regularity} by $r' = r + k - 1$), and $s$ represents a new kind of transient.
If $s \neq 0$, the recurrence does not look back to the beginning of $\word_{5/4}$.
Toward Equation~\eqref{eqn:double transient general form}, we examine the termwise difference of $w(6 i + 163183)_{i \geq 0}$ and $w(i + 12607)_{i \geq 0}$, which is
\[
	2, 3, 2, 3, 2, 1, 2, 1, 2, 3, 2, 3, 2, 1, 2, 1, 2, 3, 2, 3, 2, 1, 2, 1, \dots.
\]
This sequence is not constant, so the $d$ in Equation~\eqref{eqn:double transient general form} is in fact a function of $i$; however, it appears to be periodic with period length $8$.
Moreover, although we obtained this sequence by looking at positions of $8$ and $6$, in fact this periodic difference begins $6687$ terms earlier, before the first occurrences of $8$ and $6$.
This gives Corollary~\ref{cor:5over4}, which in turn suggests the definition of $\varphi$ in Notation~\ref{not:phi}.
The 8 residue classes in Corollary~\ref{cor:5over4} correspond to the 8 types of letters.
Note there is some flexibility in the definition of $\varphi$.
To parallel morphisms in~\cite{Guay--Shallit, Pudwell--Rowland,Rowland--Shallit}, we have chosen $\varphi$ such that the last letter in $\varphi(n_j)$ depends on $n$.

\section{Basic properties of the words $\p$ and $\z$}\label{sec:first-prop}

The following definition is motivated by the morphism $\varphi$ in Notation~\ref{not:phi}, where the subscripts increase by $1$ modulo $8$.

\begin{definition}
A (finite or infinite) word $w$ on $\Sigma_8$ is \emph{subscript-increasing} if the subscripts of the letters of $w$ increase by $1$ modulo $8$ from one letter to the next.
\end{definition}

Note that if $w$ is a subscript-increasing word on $\Sigma_8$, then so is $\varphi(w)$.
For every subscript-increasing word $w$ on $\Sigma_8$, it follows from Notation~\ref{not:phi} that the subsequence of letters with even subscripts in $\varphi(w)$ is a factor of $(0_0 0_2 1_4 1_6)^\omega$.

Iterating $\varphi$ on any nonempty word on $\Sigma_8$ will eventually give a word containing letters $n_j$ with arbitrarily larger $n$.
Indeed, after one iteration, we see a letter with subscript $3$ or $7$, so after two iterations we see a letter with subscript $7$.
Since $\varphi(n_7)$ contains $(n + 2)_7$, the alphabet grows without bound.

Before position $6764$, we cannot expect the prefix of $\word_{5/4}$ to be the image of another word under the morphism $\varphi$ because the five columns have not become periodic yet (recall Figure~\ref{fig:5over4 - 6} where $w(6759)$ is the last term of $w(6i+3)$ before a periodic pattern appears).

Starting at position $6764$, the suffix $w(6764) w(6765) \cdots$ of $\word_{5/4}$ is $111003011012 \cdots$.
By Notation~\ref{not:phi}, there is a unique way to assign subscripts to these letters to obtain an image of a word under $\varphi$, namely $1_4 1_5 1_6 0_7 0_0 3_1 0_2 1_3 1_4 0_5 1_6 2_7 \cdots$.
There are two subscript-increasing preimages of this word under $\varphi$, namely
\begin{align*}
	w(6764)_4 w(6764)_5 \cdots
	&= \varphi ( 0_2 0_3 3_4 0_5 1_6 1_7 (-1_0) 2_1 0_2 2_3 2_4 0_5 \cdots ) \\
	&=\varphi ( 2_6 0_7 1_0 0_1 (-1_2) 1_3 1_4 2_5 2_6 2_7 0_0 0_1 \cdots ).
\end{align*}
The preimage $2_6 0_7 1_0 \cdots$ appears to contain infinitely many letters of the form $-1_j$, whereas the preimage $0_2 0_3 3_4 \cdots$ does not.
Since we would like to further de-substitute a suffix of one of these preimages, we choose the preimage $0_2 0_3 3_4 \cdots$ in the following definition.
(In fact we will determine the structure of $0_2 0_3 3_4 \cdots$, and this will imply that the preimage $2_6 0_7 1_0 \cdots$ does contain infinitely many letters of the form $-1_j$.)

\begin{definition}\label{p and z definition}
Let $\p$ denote the length-$6764$ prefix of $\word_{5/4}$.
We define the word
\[
\z = 0_2 0_3 3_4 0_5 1_6 1_7 (-1_0) 2_1 0_2 2_3 2_4 0_5 \cdots 0_0 1_1 0_2 0_3 1_4 2_5 1_6 2_7 0_0 0_1 0_2 3_3
\]
to be the length-$20226$ subscript-increasing word on $\Sigma_8$ starting with $0_2$ and satisfying
\[
\tau(\varphi(\z))=w(6764)w(6765) \cdots w(6764+6\size{\z}-1).
\]
We also define $\s=\z\varphi(\z)\varphi^2(\z)\cdots$, which is a subscript-increasing infinite word on the alphabet $\Sigma_8$.
\end{definition}

The following lemma states several properties of $\p$, $\z$, and $\s$.

\begin{lemma}\label{lem:various-properties}
Let $\Gamma \subset \Sigma_8$ be the finite alphabet
\[
\{ -3_0,-3_2,-2_0,-2_1,-2_2,-2_3,-2_5,-2_7,-1_1,-1_3,-1_4,-1_5,-1_6
 ,-1_7,0_4,0_6\}.
\]
We have the following properties.
\begin{enumerate}

\item\label{common-suffix-p-tau-z}
The length-$844$ suffixes of $\p$ and $\tau(\z)$ are equal.

\item\label{stuff about z}
The word $\z$ is a subscript-increasing finite word whose alphabet is the $32$-letter set
\begin{align*}
\Alphabet(\z) =
\{&
{-1}_0, {-1}_2, 0_0, 0_1, 0_2, 0_3, 0_5, 0_7, 1_0, 1_1,
1_2, 1_3, 1_4, 1_5, 1_6, 1_7, \\
&2_1, 2_3, 2_4, 2_5, 2_6, 2_7, 3_1, 3_3, 3_4, 3_5, 3_6, 3_7, 4_1, 4_3, 4_5, 4_7
\}.
\end{align*}
In particular, $\z$ is a word on the alphabet $\Sigma_8\setminus \Gamma$.
The last letter of the form $-1_j$ in $\z$ appears in position $80$.

\item\label{stuff about s}
The word $\s$ is a subscript-increasing infinite word on $\Sigma_8\setminus \Gamma$.
Moreover, the subsequence of letters in $\s$ with even subscripts starting at position $86$ is $(0_0 0_2 1_4 1_6)^\omega$.

\item\label{even-subscript}
For all words $w$ on $\Sigma_8$, the set of letters with even subscripts in $\varphi(w)$ is a subset of $\{0_0, 0_2, 1_4, 1_6\}$.
Moreover, if $w$ is subscript-increasing, then the subsequence of letters with even subscripts in $\varphi(w)$ is a factor of $(0_0 0_2 1_4 1_6)^\omega$.

\item\label{last-letter-phi}
For each $n_j\in \Sigma_8\setminus\Gamma $, the last letter of $\varphi(n_j)$ is not of the form $0_i$ or $1_i$.

\end{enumerate}
\end{lemma}

\begin{proof}
Parts~\ref{common-suffix-p-tau-z} and \ref{stuff about z} follow from computing $\p$ and $\z$.
Part~\ref{even-subscript} follows from inspection of $\varphi$.
To see Part~\ref{last-letter-phi}, for each $j$ we set the last letter $(n+d)_i$ of $\varphi(n_j)$ equal to $0_i$ and $1_i$, solve each for $n$, and observe that $n_j\in\Gamma$.

For Part~\ref{stuff about s}, recall that $\z$ is a prefix of $\s$.
First we prove that $\s$ is a subscript-increasing infinite word on $\Sigma_8\setminus \Gamma$.
The letters $5_3$ and $6_3$ arise from letters $3_1$, $3_5$, $4_1$, and $4_5$ in $\z$.
As previously mentioned, letters with subscript $7$ appear when iterating $\varphi$ on any nonempty word.
One checks by induction that these are the only letters in $\s$.
Therefore the alphabet of $\s$ is $\Alphabet(\z) \cup \{5_3, 6_3\} \cup \{ n_7 \colon n\ge 5 \}$.
To show that $\s$ is subscript-increasing, note that $\z$ ends with $3_3$ and $\varphi(\z)$ starts with $\varphi(0_2)=1_4 \cdots$.
The other boundaries follow inductively by applying $\varphi$.

Next we show that the subsequence of letters in $\s$ with even subscripts starting at position $86$ is $(0_0 0_2 1_4 1_6)^\omega$.
We check that the subsequence of letters in $\z$ with even subscripts starting at position $86$ is a finite prefix of $(0_0 0_2 1_4 1_6)^\omega$.
By Part~\ref{even-subscript}, the sequence of letters in $\varphi(\z) \varphi^2(\z) \cdots = 1_4 1_5 \cdots$ with even subscripts is a factor of $(0_0 0_2 1_4 1_6)^\omega$.
Since $\z = \cdots 0_2 3_3$, the claim follows.
\end{proof}

\section{Lengths of $5/4$-powers}\label{sec:length}

Pudwell and Rowland introduced the notion of locating lengths as a tool to prove that morphisms preserve the property of $a/b$-power-freeness~\cite{Pudwell--Rowland}.
We use this notion in Section~\ref{sec:pre-5/4-power-freeness} to show that $\varphi$ has a weaker property.

\begin{definition}
Let $k \ge 2$ and $\ell \ge 1$.
Let $\mu$ be a $k$-uniform morphism on an alphabet $\Sigma$.
We say that $\mu$ \emph{locates words of length $\ell$} if for each word $u$ of length $\ell$ on $\Sigma$ there exists an integer $m$ such that, for all $v\in \Sigma^*$, every occurrence of the factor $u$ in $\mu(v)$ begins at a position congruent to $m$ modulo $k$.
\end{definition}

If $\mu$ locates words of length $\ell$, then $\mu$ also locates words of length $\ell+1$, since if $\size{u} = \ell +1$ then the positions of the length-$\ell$ prefix of $u$ in an image under $\mu$ is determined modulo $k$.

\begin{lemma}\label{lem:Phi-Locates-6-Length}
The $6$-uniform morphism $\varphi : \Sigma_8^* \to \Sigma_8^*$ locates words of length $6$.
\end{lemma}

\begin{proof}
Let $v$ be a word in $\Sigma_8^*$.
Note that $v$ is not necessarily subscript-increasing.
We look at occurrences of length-$6$ factors in $\varphi(v)$.
There are eight cases, depending on the subscript of the initial letter.
We write out the details for the subscript $2$.
There are four possible forms for the length-$6$ factor in this case.
The other cases are analogous, some of which involve five possible forms.

Consider a length-$6$ factor of $\varphi(v)$ whose initial letter is of the form $n_2$.
Then in fact the initial letter is $0_2$, since this is the only letter in $\varphi(v)$ with subscript $2$.
There are four forms depending on which column the letter $0_2$ is, namely
\[
\begin{array}{cccccccc}
u & = & 0_2 & 0_3 & 1_4 & (n+3)_5 & \cdots \\
\bar{u} & = & 0_2 & 0_3 & 1_4 & (\bar{n}+1)_5 & \cdots\\
\bar{\bar{u}} & = & 0_2 & (\bar{\bar{n}}+2)_3 & \cdots \\
\bar{\bar{\bar{u}}} & = & 0_2 & 1_3 & 1_4 & 0_5 & 1_6 & (\bar{\bar{\bar{n}}}+2)_7.
\end{array}
\]
If $0_2$ appears in the third column, then the definition of $\varphi$ implies that the factor is of the form $u$ or $\bar{u}$.
If $0_2$ appears in the fifth column, then it is of the form $\bar{\bar{u}}$.
If $0_2$ appears in the first column, then it is of the form $\bar{\bar{\bar{u}}}$.
For each pair of these factors, we show that either they are unequal or they occur at positions that are equivalent modulo $6$.

If $u=\bar{u}$, then they only occur at positions that are equivalent modulo $6$ because they are in the same column.

We directly see that the letters with subscript $3$ in $u$ (resp., $\bar{u}$) and $\bar{\bar{\bar{u}}}$ do not match, so they are different length-$6$ factors.

Next we compare $u$ (resp., $\bar{u}$) with $\bar{\bar{u}}$.
For them to be equal, we have to attach, in $\bar{\bar{u}}$, a prefix of one of the length-$6$ images of $\varphi$ starting with $1_4$, namely $\varphi(m_2)$ or $\varphi(m_6)$.
Both cases give
\[
\begin{array}{ccccccccc}
u & = & 0_2 & 0_3 & 1_4 & (n+3)_5 & \cdots \\
\bar{u} & = & 0_2 & 0_3 & 1_4 & (\bar{n}+1)_5 & \cdots \\
\bar{\bar{u}} & = & 0_2 & (\bar{\bar{n}}+2)_3 & 1_4 & 1_5 & 1_6 & 0_7.
\end{array}
\]
For $u$ or $\bar{u}$ to be equal to $\bar{\bar{u}}$, we must have
\[
\begin{array}{cccccccc}
u & = & 0_2 & 0_3 & 1_4 & (n+3)_5 & 1_6 & 1_7\phantom{,} \\
\bar{u} & = & 0_2 & 0_3 & 1_4 & (\bar{n}+1)_5 & 1_6 & 1_7\phantom{,} \\
\bar{\bar{u}} & = & 0_2 & (\bar{\bar{n}}+2)_3 & 1_4 & 1_5 & 1_6 & 0_7.
\end{array}
\]
We now see that the letters with subscript $7$ of $u$ (resp., $\bar{u}$) and $\bar{\bar{u}}$ are not the same, so they are different length-$6$ factors.

Finally we compare $\bar{\bar{u}}$ with $\bar{\bar{\bar{u}}}$.
For them to be equal, the same argument shows
\[
\begin{array}{cccccccc}
\bar{\bar{u}} & = & 0_2 & (\bar{\bar{n}}+2)_3 & 1_4 & 1_5 & 1_6 & 0_7 \\
\bar{\bar{\bar{u}}} & = & 0_2 & 1_3 & 1_4 & 0_5 & 1_6 & (\bar{\bar{\bar{n}}}+2)_7.\\
\end{array}
\]
The letters with subscript $5$ of $\bar{\bar{u}}$ and $\bar{\bar{\bar{u}}}$ do not match, so they are different length-$6$ factors.
\end{proof}

\begin{lemma}\label{lem:5-4-powers-length-dvi-by-6}
Let $\ell\ge 6$ be an integer.
Let $w$ be a subscript-increasing word on $\Sigma_8$.
If $\varphi(w)$ contains a $5/4$-power of length $5\ell$, then $\ell$ is divisible by $6$.
\end{lemma}

\begin{proof}
Assume that $\varphi(w)$ contains a $5/4$-power $xyx$ with $\lvert x \rvert = \ell$ and $\lvert y \rvert = 3\ell$.
Since the first letters of the two occurrences of $x$ are the same, their subscripts are equal.
Since $w$ is subscript-increasing by assumption, then $\varphi(w)$ is subscript-increasing, so $\lvert xy \rvert$ is divisible by $8$.
Since $\lvert xy \rvert = 4\ell$, this implies $\ell$ is even.
Since $\size{x} \geq 6$, Lemma~\ref{lem:Phi-Locates-6-Length} implies that $\lvert xy \rvert$ is also divisible by $6$.
Consequently, $\lvert xy \rvert$ is divisible by $24$, so $\ell$ is divisible by $6$.
\end{proof}

\section{Pre-$5/4$-power-freeness}\label{sec:pre-5/4-power-freeness}

A morphism $\mu$ on an alphabet $\Sigma$ is \emph{$a/b$-power-free} if $\mu$ preserves $a/b$-power-freeness, that is, for all $a/b$-power-free words $w$ on $\Sigma$, $\mu(w)$ is also $a/b$-power-free.
Previously studied words $\word_{a/b}$~\cite{Guay--Shallit, Rowland--Shallit, Pudwell--Rowland} have all been described by $a/b$-power-free morphisms.
However, the morphism $\varphi$ defined in Notation~\ref{not:phi} is not $5/4$-power-free.
Indeed for any integers $n,\bar{n} \in \Z$, the word $0_4 n_5 \bar{n}_6$ is $5/4$-power-free, but $\varphi(0_4 n_5 \bar{n}_6)$ contains the length-$10$ factor
\[
1_4 1_5 1_6 1_7 0_0 0_1 0_2 (n+2)_3 1_4 1_5,
\]
which is a $5/4$-power.
Therefore, to prove that $\word_{5/4}$ is $5/4$-power-free, we use a different approach.
We still need to guarantee that there are no $5/4$-powers in certain images $\varphi(w)$.
Specifically, we would like all factors $xyx'$ of $w$ with $\size{x}=\frac{1}{3}\size{y}=\size{x'}$ to satisfy $\varphi(x) \neq \varphi(x')$.
We use the following concept.

\begin{definition}
A word $w$ on $\Sigma_8$ is a \emph{pre-$5/4$-power} if $\varphi(w)$ is a $5/4$-power.
\end{definition}

A nonempty word $w$ on $\Sigma_8$ is a pre-$5/4$-power if and only if $w = x y x'$ for some $x, y, x'$ with $\size{x}=\frac{1}{3}\size{y}=\size{x'}$ such that $\varphi(x) = \varphi(x')$.
The next lemma follows from the definition of $\varphi$.

\begin{lemma}\label{lem:equality}
Two letters $n_i, \bar{n}_j \in \Sigma_8$ satisfy $\varphi(n_i) = \varphi(\bar{n}_j)$ if and only if $i - j \in \{-4, 0, 4\}$ and
\[
	n - \bar{n} =
	\begin{cases}
		\frac{1}{2} (i - j)	& \text{if $i$ is even} \\
		0			& \text{if $i$ is odd}.
	\end{cases}
\]
\end{lemma}

We have the following characterization of subscript-increasing pre-$5/4$-powers.

\begin{lemma}\label{lem:pre-5-4-power}
A nonempty subscript-increasing word $w$ on $\Sigma_8$ is a pre-$5/4$-power if and only if $w=x y x'$ for some $x, y, x'$ with $\size{x}=\frac{1}{3}\size{y}=\size{x'}$ such that
\begin{enumerate}
\item\label{equal condition}
if the sequences of subscripts in $x$ and $x'$ are equal, then $x = x'$, and
\item\label{differ by 4 condition}
if the sequences of subscripts in $x$ and $x'$ differ by $4$, then the $m$th letters $x(m)$ and $x'(m)$ satisfy
\[
	\tau(x(m)) - \tau(x'(m)) \in
	\begin{cases}
		\{-2, 2\}	& \text{if the subscript of $x(m)$ is even} \\
		\{0\}		& \text{if the subscript of $x(m)$ is odd}
	\end{cases}
\]
for all $m \in \{0, 1, \dots, \size{x} - 1\}$.
\end{enumerate}
\end{lemma}

\begin{proof}
Let $x, y, x'$ be nonempty words on $\Sigma_8$ such that $x y x'$ is subscript-increasing, $\size{x}=\frac{1}{3}\size{y}=\size{x'}$, and $\varphi(x) = \varphi(x')$.
Since $\size{x y} = 4 \size{x}$, the sequences of subscripts in $x$ and $x'$ are equal or differ by $4$.
By Lemma~\ref{lem:equality}, $x y x'$ is a pre-$5/4$-power if and only if, for each $m \in \{0, 1, \dots, \size{x} - 1\}$, we have
\[
	\tau(x(m)) - \tau(x'(m)) \in
	\begin{cases}
		\{0\}		& \text{if their subscripts are equal} \\
		\{-2, 2\}	& \text{if their subscripts are even and differ by $4$} \\
		\{0\}		& \text{if their subscripts are odd and differ by $4$}.
	\end{cases}
\]
This is equivalent to Conditions~\ref{equal condition} and \ref{differ by 4 condition} in the statement.
\end{proof}

For example, the word $0_0n_1\bar{n}_2 \bar{\bar{n}}_32_4$ is a pre-$5/4$-power because $\tau(0_0) - \tau(2_4) = 0 - 2 \in \{2, -2\}$; indeed
\[
\varphi(0_0n_1\bar{n}_2 \bar{\bar{n}}_32_4)= 0_0 1_1 0_2 0_3 1_4 3_5 \varphi(n_1\bar{n}_2 \bar{\bar{n}}_3) 0_0 1_1 0_2 0_3 1_4 3_5
\]
is a $5/4$-power of length $30$.
On the other hand, the word $0_0n_1\bar{n}_2 \bar{\bar{n}}_30_4$ is not a pre-$5/4$-power because $\tau(0_0) - \tau(0_4) = 0 \notin \{2, -2\}$; indeed
\[
\varphi(0_0n_1\bar{n}_2 \bar{\bar{n}}_30_4)= 0_0 1_1 0_2 0_3 1_4 3_5 \varphi(n_1\bar{n}_2 \bar{\bar{n}}_3) 0_0 1_1 0_2 0_3 1_4 1_5
\]
is not a $5/4$-power.
Similarly, $0_1 n_2 \bar{n}_3 \bar{\bar{n}}_4 2_5$ is not a pre-$5/4$-power, because $\tau(0_1) - \tau(2_5) = -2 \notin \{0\}$; indeed
\[
	\varphi(0_1 n_2 \bar{n}_3 \bar{\bar{n}}_4 2_5)
	= 1_6 1_7 0_0 0_1 0_2 2_3 \varphi(n_1 \bar{n}_2 \bar{\bar{n}}_3) 1_6 1_7 0_0 0_1 0_2 4_3
\]
is not a $5/4$-power.

In addition to not being a pre-$5/4$-power, the word $0_0n_1\bar{n}_2 \bar{\bar{n}}_30_4$ is $5/4$-power-free since $0_0$ and $0_4$ are different letters.

\begin{proposition}\label{pro:pre-5-4-power-strong-cond}
Every $5/4$-power on $\Sigma_8$ is a pre-$5/4$-power.
\end{proposition}

\begin{proof}
Let $x, y$ be nonempty words on $\Sigma_8$ with $\size{x} = \frac{1}{3} \size{y}$, so that $x y x$ is a $5/4$-power.
Then $\varphi(x y x) = \varphi(x) \varphi(y) \varphi(x)$ is a $5/4$-power, so $x y x$ is a pre-$5/4$-power.
\end{proof}

Proposition~\ref{pro:pre-5-4-power-strong-cond} implies that if a word $w$ is pre-$5/4$-power-free then $w$ is $5/4$-power-free.

Let $\Gamma$ be the alphabet in Lemma~\ref{lem:various-properties}.
For all subscript-increasing words $w$ on $\Sigma_8\setminus\Gamma$, the combination of Lemmas~\ref{lem:5-4-powers-larger-6} and \ref{lem:5-4-powers-less-than-6} shows that, if $w$ is pre-$5/4$-power-free, then $\varphi(w)$ is $5/4$-power-free.
(We include Lemma~\ref{lem:5-4-powers-less-than-6} because it complements Lemma~\ref{lem:5-4-powers-larger-6}, even though we will not use it to prove that $\p\tau(\varphi(\s))$ is $5/4$-power-free.)

\begin{lemma}\label{lem:5-4-powers-larger-6}
If $w$ is a pre-$5/4$-power-free subscript-increasing word on $\Sigma_8$, then $\varphi(w)$ contains no $5/4$-power of length greater than or equal to $30$.
\end{lemma}

\begin{proof}
Proceed toward a contradiction and assume that there exists a $5/4$-power in $\varphi(w)$ of the form $xyx$ with $\size{x} = \ell$ and $\size{y} = 3\ell$ and $\ell \geq 6$.
Let $j$ be the initial position of $xyx$ in $\varphi(w)$.
Write $j=6i_1+r$ with $0\le r \le 5$.
Since $w$ is subscript-increasing, Lemma~\ref{lem:5-4-powers-length-dvi-by-6} implies that $\ell=\size{x}$ is divisible by $6$.
Then $y$ begins at position $j+ \size{x} = 6i_2+r$ and the second occurrence of $x$ begins at position $j+ \size{xy} = 6i_3+r$ for some $i_2, i_3$.
Now we shift if necessary so that $r = 0$; let $x'y'x'$ be the word of length $\size{xyx}$ starting at position $6i_1$ in $\varphi(w)$.
Let $uvu'$ be the factor of $w$ of length $\frac{1}{6} \size{x y x}$ starting at position $i_1$ such that $\size{u}=\frac{1}{3}\size{v}=\size{u'}$.
Then $\varphi(uvu') = x'y'x'$, so $uvu'$ is a pre-$5/4$-power.
This contradicts the hypothesis that $w$ is a pre-$5/4$-power-free word.
\end{proof}

\begin{lemma}\label{lem:5-4-powers-less-than-6}
Let $\Gamma$ be the alphabet in Lemma~\ref{lem:various-properties}.
If $w$ is a subscript-increasing word on $\Sigma_8\setminus\Gamma$, then $\varphi(w)$ contains no $5/4$-power of length less than or equal to $25$.
\end{lemma}

\begin{proof}
Given $(n+d)_i$ where $d\in \{1,2,3\}$ and $0\le i \le 7$, we will need to know the values of $j\in\{0,1,\dots,7\}$ for which the letter $(n+d)_i$ is the last letter of $\varphi(n_j)$.
This is given by the following table.
\begin{equation}\label{eqn:phi-inverse-lookup}
\begin{array}{c|cccccc}
(d,i) & (1,1) & (1,5) & (2,3) & (2,7) & (3,1) & (3,5) \\
\hline
j & 6 & 4 & 1,5 & 3,7 & 2 & 0
\end{array}
\end{equation}

As in the proof of Lemma~\ref{lem:5-4-powers-length-dvi-by-6}, $\ell$ is even.
It suffices to look at $\ell=2$ and $\ell=4$.
Since $w$ is subscript-increasing, we may consider the word
\[
\varphi(n_0) \varphi(n_1) \varphi(n_2) \varphi(n_3) \varphi(n_4) \varphi(n_5)
\varphi(n_6) \varphi(n_7)
\]
circularly and slide a window of length $5\ell$ through this word.
Here $n$ is a symbol, not an integer.
For each factor of length $5 \ell$, we compare its prefix of length $\ell$ to its suffix of length $\ell$.
If they are elements of $\Sigma_8^*$ (that is, they do not involve $n$), then we check that they are unequal.
Otherwise, for each pair of letters that involves $n$, we solve for $n$, and we use Table~\eqref{eqn:phi-inverse-lookup} to determine the possible subscripts $j$ of $n$.
The set $\Gamma$ is precisely the set of letters $n_j$ that arise.
\end{proof}

As an example of the algorithm described in the previous proof, for $\ell=2$ one possible form of $5/4$-powers of length $10$ is
\[
0_2 (n + 2)_3 1_4 1_5 1_6 0_7 0_0 (\bar{n} + 3)_1 0_2 1_3.
\]
We solve $(n + 2)_3 = 1_3$ and get $n=-1$.
Since $d=2$ and the subscript is $i=3$, then according to Table~\eqref{eqn:phi-inverse-lookup} we find $j\in\{1,5\}$.
Therefore $\varphi(-1_1 \bar{n}_2 \bar{\bar{n}}_3)$, $\varphi(-1_1 \bar{n}_2 \bar{\bar{n}}_7)$, $\varphi(-1_5 \bar{n}_2 \bar{\bar{n}}_3)$, and $\varphi(-1_5 \bar{n}_2 \bar{\bar{n}}_7)$ all contain a $5/4$-power of length $10$ (even though $-1_1 \bar{n}_2 \bar{\bar{n}}_3$ is the only subscript-increasing preimage).
In this case, we add $-1_1$ and $-1_5$ to $\Gamma$.

We need a stronger result than Lemmas~\ref{lem:5-4-powers-larger-6} and \ref{lem:5-4-powers-less-than-6} provide.
Namely, since we iterate $\varphi$, we need that images under $\varphi$ are not just $5/4$-power-free but are in fact pre-$5/4$-power-free.

\begin{proposition}\label{pro:pre-5-4-power-phi}
Let $\Gamma$ be the alphabet in Lemma~\ref{lem:various-properties}.
If $w$ is a pre-$5/4$-power-free subscript-increasing word on $\Sigma_8\setminus\Gamma$, then $\varphi(w)$ is pre-$5/4$-power-free.
\end{proposition}

\begin{proof}
Since $w$ is subscript-increasing, so is $\varphi(w)$.
Let $xyx'$ be a nonempty factor of $\varphi(w)$ with $\size{x}=\frac{1}{3}\size{y}=\size{x'}$.
We show that $xyx'$ is not a pre-$5/4$-power.
We consider two cases depending on the parity of $\size{x}$.

\textbf{Case 1.}
Suppose that $\size{x}$ is odd.
Then the sequences of subscripts in $x$ and $x'$ differ by $4$.

If $x$ contains a letter with an even subscript, then the corresponding letter in $x'$ also has an even subscript.
This pair of corresponding letters belongs to $\{(0_0,1_4),(0_2,1_6),(1_4,0_0),(1_6,0_2)\}$ by Part~\ref{even-subscript} of Lemma~\ref{lem:various-properties}.
Therefore the images under $\tau$ of the two letters differ by $\pm 1 \notin \{-2, 2\}$, so $x y x'$ is not a pre-$5/4$-power by Lemma~\ref{lem:pre-5-4-power}.

If $x$ does not contain a letter with an even subscript, then $\size{x}=1$, so $x$ and $x'$ are letters with odd subscripts.
Then $(x,x')$ is of one of the forms
\[
\begin{array}{ccccc}
(1_1,(n+3)_5) && (0_3,1_7) && ((n+3)_5,0_1) \phantom{,}\\
(1_7,(n+2)_3) && (0_1,1_5) && ((n+2)_3,0_7)\phantom{,}\\
(1_5,(n+3)_1) && (0_7,1_3) && ((n+3)_1,0_5)\phantom{,}\\
(1_3,(n+2)_7) && (0_5,1_1) && ((n+2)_7,0_3)\phantom{,}\\
(1_1,(n+1)_5) && && ((n+1)_5,0_1)\phantom{,} \\
(1_5,(n+1)_1) && && ((n+1)_1,0_5).
\end{array}
\]
Since $w$ is a word on $\Sigma_8\setminus\Gamma$, we have $\tau(x) \neq \tau(x')$ thanks to Part~\ref{last-letter-phi} of Lemma~\ref{lem:various-properties}.
Therefore $x y x'$ is not a pre-$5/4$-power by Lemma~\ref{lem:pre-5-4-power}.

\textbf{Case 2.}
Assume that $\size{x}\ge 2$ is even.
Then the sequences of subscripts in $x$ and $x'$ are equal.
By Part~\ref{even-subscript} of Lemma~\ref{lem:various-properties}, each pair of corresponding letters in $x$ and $x'$ with even subscripts belongs to $\{(0_0,0_0), (0_2,0_2), (1_4,1_4), (1_6,1_6)\}$.
To show that $x y x'$ is not a pre-$5/4$-power, we use Lemma~\ref{lem:pre-5-4-power} and show that there exists a pair $(x(m), x'(m))$ of corresponding letters with odd subscript $j$ such that $x(m) \neq x'(m)$.
Since $\size{x}$ is even and $\size{xy} = 4\size{x}$, then $\size{xy} \equiv r \mod 6$ with $r\in\{0,2,4\}$.
We break the remainder of the proof into two cases, depending on the value of $r$.

Case 2.1.
Suppose that $\size{xy} \equiv r \mod 6$ with $r\in\{2,4\}$.
If $r=2$, then $x$ contains a letter in the second (respectively, fourth or sixth) column if and only if $x'$ contains the corresponding letter in the fourth (respectively, sixth or second) column.
If $r=4$, then $x$ contains a letter in the second (respectively, fourth or sixth) column if and only if $x'$ contains the corresponding letter in the sixth (respectively, second or fourth) column.
Therefore it suffices to compare the second, fourth, and sixth columns, and we have pairs of the forms
\[
(1_j,0_j), \, (1_j,(n+d)_j), \, (0_j,1_j), \, (0_j,(n+d)_j), \, ((n+d)_j,1_j), \, ((n+d)_j,0_j)
\]
with $d\in\{1,2,3\}$.
Since $w$ is a word on $\Sigma_8\setminus\Gamma$, we have $x(m) \neq x'(m)$ thanks to Part~\ref{last-letter-phi} of Lemma~\ref{lem:various-properties}.
Therefore $x y x'$ is not a pre-$5/4$-power by Lemma~\ref{lem:pre-5-4-power}.

Case 2.2.
Assume that $\size{xy} \equiv 0 \mod 6$, which implies $\size{x}\ge 6$ since $\size{xy}=4\size{x}$ and $\size{x}$ is even.
Since $\size{xy} \equiv 0 \mod 6$, $x$ and $x'$ agree on their backgrounds, that is, letters in the first five columns.
For instance, we may have
\begin{align*}
	x &= 1_1 0_2 0_3 1_4 (n + 3)_5 1_6 \\
	y &= 1_7 0_0 0_1 0_2 (\bar{n} + 2)_3 1_4 1_5 1_6 0_7 0_0 (\bar{\bar{n}} + 3)_1 0_2 1_3 1_4 0_5 1_6 (\bar{\bar{\bar{n}}} + 2)_7 0_0 \\
	x' &= 1_1 0_2 0_3 1_4 (\bar{\bar{\bar{\bar{n}}}} + 1)_5 1_6
\end{align*}
as a factor of $\varphi(n_0 \bar{n}_1 \cdots \bar{\bar{\bar{\bar{\bar{n}}}}}_5)$.
Recall that we need to exhibit a pair of corresponding letters in $x$ and $x'$ with odd subscript $j$ such that $x(m) \neq x'(m)$.
The only possibility is that $x(m)$ and $x'(m)$ both belong to the sixth column.
Toward a contradiction, suppose that such a pair does not exist, and so $x$ and $x'$ also agree on their letters belonging to the sixth column.
In particular, $x=x'$.
We have thus found a $5/4$-power in $\varphi(w)$.
This violates Lemma~\ref{lem:5-4-powers-larger-6} since $w$ is pre-$5/4$-power-free by assumption.
Therefore there is a pair with $x(m) \neq x'(m)$, so $x y x'$ is not a pre-$5/4$-power by Lemma~\ref{lem:pre-5-4-power}.
\end{proof}

Now we return to studying the particular words $\z$ and $\s =\z \varphi(\z) \varphi^2(\z) \cdots$ from Definition~\ref{p and z definition}.
First we show that $\s$ is pre-$5/4$-power-free.
As a consequence, we will show that $\p\tau(\varphi(\s))$ is $5/4$-power-free in Section~\ref{sec:5/4-power-freeness}.
The words $\p$ and $\tau(\z)$ have a common suffix $0003$.
Since this suffix is a factor of $\tau(\varphi(1_1))=\tau(\varphi(1_5))$, this requires extra consideration in Theorem~\ref{thm:s-pre-5-4-power-free} and several other results.

\begin{theorem}\label{thm:s-pre-5-4-power-free}
The word $\s$ is pre-$5/4$-power-free.
\end{theorem}

\begin{proof}
We show that, for all $e\ge 1$, $\z\varphi(\z) \cdots \varphi^e(\z)$ is pre-$5/4$-power-free.
Note that $\z\varphi(\z) \cdots \varphi^e(\z) \in(\Sigma_8\setminus\Gamma)^*$ for all $e\ge 1$.
Since $\s$ is subscript-increasing by Part~\ref{stuff about s} of Lemma~\ref{lem:various-properties}, its factor $\z\varphi(\z) \cdots \varphi^e(\z)$ is also subscript-increasing.
We proceed by induction on $e$.

For $e = 1$, one checks programmatically that $\z\varphi(\z)$ is pre-$5/4$-power-free.
Our implementation took about $6$ hours (although this could be reduced by parallelizing).

Now suppose that the word $\z\varphi(\z) \cdots \varphi^e(\z)$ is pre-$5/4$-power-free.
We show that $\z\varphi(\z) \cdots \varphi^e(\z) \varphi^{e+1}(\z)$ is also pre-$5/4$-power-free.
Since $\z\varphi(\z) \cdots \varphi^e(\z)$ is a word on $\Sigma_8\setminus\Gamma$, Proposition~\ref{pro:pre-5-4-power-phi} implies that $\varphi(\z) \varphi^2(\z) \cdots \varphi^{e + 1}(\z)$ is pre-$5/4$-power-free.
Also, $\z$ is pre-$5/4$-power-free since $\z\varphi(\z)$ is pre-$5/4$-power-free.
Therefore it suffices to check factors of $\z\varphi(\z) \cdots \varphi^{e+1}(\z)$ that overlap both $\z$ and $\varphi(\z) \cdots \varphi^{e+1}(\z)$.
Let $xyx'$ be such a factor, with $\size{x}=\frac{1}{3}\size{y}=\size{x'}$.

If $x$ is a factor of $\z$, then
\[
\size{xyx'} = 5 \size{x} \le 5 \size{\z} < 6 \size{\z} = \size{\varphi(\z)}.
\]
Therefore $xyx'$ is a factor of $\z \varphi(\z)$.
The base case of the induction implies that $x y x'$ is not a pre-$5/4$-power.
(Note that if we had used $e=0$ as the base case of the induction, we still would have needed to check the case $e=1$ programmatically here.)

If $x$ is not a factor of $\z$, then $x$ overlaps the last letter of $\z$ and the first letter of $\varphi(\z) \cdots \varphi^{e+1}(\z)$, since we assume $x y x'$ overlaps both $\z$ and $\varphi(\z) \cdots \varphi^{e+1}(\z)$.
There are two cases.

If $x$ overlaps the last $5$ letters of $\z$, then $x$ contains the suffix $2_70_00_10_23_3$ of $\z$.
If the subscripts in $x$ and $x'$ differ by $4$, then $x'$ being a factor of $\varphi(\z) \cdots \varphi^{e+1}(\z)$ implies that the factor $2_70_00_10_23_3$ of $x$ corresponds to a factor $n_31_4\bar{n}_51_6\bar{\bar{n}}_7$ of $x'$.
So $x y x'$ is not a pre-$5/4$-power by Lemma~\ref{lem:pre-5-4-power}.
If the subscripts in $x$ and $x'$ line up, then the factor $2_70_00_10_23_3$ of $x$ corresponds to a factor $n_7 0_0 \bar{n}_1 0_2 \bar{\bar{n}}_3$ of $x'$.
Since $2_70_00_10_23_3$ is not a factor of $\varphi(\z)\cdots \varphi^{e+1}(\z)$, we must have $n_7 \neq 2_7$ or $\bar{n}_1\neq 0_1$ or $\bar{\bar{n}}_3 \neq 3_3$.
Therefore $x y x'$ is not a pre-$5/4$-power by Lemma~\ref{lem:pre-5-4-power}.

Suppose $x$ overlaps fewer than the last $5$ letters of $\z$.
If $\size{x}$ is odd, then the subscripts in $x$ and $x'$ differ by $4$.
Since $x$ contains the factor $3_31_4$, then $x'$ being a factor of $\varphi(\z) \cdots \varphi^{e+1}(\z)$ implies that the factor $3_31_4$ of $x$ corresponds to a factor $n_70_0$ of $x'$, and $xyx'$ is not a pre-$5/4$-power by Lemma~\ref{lem:pre-5-4-power}.
If $\size{x}$ is even, then the subscripts in $x$ and $x'$ line up.
The words $x$ and $x'$ agree on even subscripts (by Part~\ref{stuff about s} of Lemma~\ref{lem:various-properties} because the length-$4$ suffix of $\z$ is $0_00_10_23_3$).
For odd subscripts, if the corresponding letters of $x$ and $x'$ belong to different columns (that is, their positions are not congruent modulo $6$), then, as in Case~2.1 in the proof of Proposition~\ref{pro:pre-5-4-power-phi}, they form one of the pairs
\[
(1_j,0_j), \, (1_j,(n+d)_j), \, (0_j,1_j), \, (0_j,(n+d)_j), \, ((n+d)_j,1_j), \, ((n+d)_j,0_j)
\]
with $j$ odd and $d\in\{1,2,3\}$.
Part~\ref{last-letter-phi} of Lemma~\ref{lem:various-properties} implies that $xyx'$ is not a pre-$5/4$-power by Lemma~\ref{lem:pre-5-4-power}.
If the corresponding letters belong to the same column, then $x$ and $x'$ agree everywhere except maybe in the sixth column.
Let $j$ be the initial position of $x'$ in $\varphi(\z) \varphi^2(\z) \cdots \varphi^{e+1}(\z)$.
Let $u'$ be the word of minimal length starting at position $\lfloor\frac{j}{6}\rfloor$ in $\z \varphi(\z) \cdots \varphi^e(\z)$ such that $x'$ is a factor of $\varphi(u')$.

If $\size{u'}\le 3$, then $\size{x}\le \size{\varphi(u')} \le 18$, so $\size{xyx'} = 5\size{x} \le 90 < \size{\z\varphi(\z)}$, which means that $xyx'$ is a factor of $\z\varphi(\z)$.
Due to the base case, we already know that $xyx'$ is not a pre-$5/4$-power.

Finally, consider the case $\size{u'}\ge 4$.
Toward a contradiction, suppose $x=x'$.
Since $x = x'$ is a prefix of
\[
	0_0 0_1 0_2 3_3 \varphi(\z) \varphi^2(\z) \cdots
	=
	0_0 0_1 0_2 3_3 \cdot
	1_4 1_5 1_6 0_7 0_0 3_1 \cdot
	0_2 1_3 1_4 0_5 1_6 2_7 \cdot
	0_0 1_1 0_2 0_3 1_4 4_5
	\cdots,
\]
the word $u'$ is one of the two preimages
\[
	u'=
	\begin{cases}
	1_1 0_2 0_3 3_4 \cdots \\
	1_5 2_6 0_7 1_0 \cdots.
	\end{cases}
\]
Recall that $u'$ is a factor of $\z \varphi(\z) \cdots \varphi^e(\z)$.
By Part~\ref{stuff about s} of Lemma~\ref{lem:various-properties}, the fourth letters $3_4$ and $1_0$ do not occur in $\varphi(\z) \cdots \varphi^e(\z)$, so they must occur in $\z$.
We consider the positions where they occur.
The letter $3_4$ occurs in $\z$ only in positions $2$ and $66$.
Position $2$ is too early for $1_1 0_2 0_3 3_4$ to be a factor.
At position $66$ we have $\z(63)\z(64)\z(65)\z(66) = 0_1 0_2 0_3 3_4$.
Similarly, the letter $1_0$ occurs in $\z$ only in positions $22$ and $54$ and $78$, and we find
\begin{align*}
	\z(19)\z(20)\z(21)\z(22)&=2_5 2_6 0_7 1_0 \\
	\z(51)\z(52)\z(53)\z(54)&=1_5 2_6 2_7 1_0 \\
	\z(75)\z(76)\z(77)\z(78)&=1_5 2_6 1_7 1_0.
\end{align*}
Therefore neither $1_1 0_2 0_3 3_4$ nor $1_5 2_6 0_7 1_0$ is a factor of $\z$.
It follows that $x \neq x'$.
Since the sequences of subscripts in $x$ and $x'$ are equal, this implies that $x y x'$ is not a pre-$5/4$-power by Lemma~\ref{lem:pre-5-4-power}.
\end{proof}

\section{$5/4$-power-freeness}\label{sec:5/4-power-freeness}

In this section we show that $\p \tau(\varphi(\s))$ is $5/4$-power-free.
As a consequence of Theorem~\ref{thm:s-pre-5-4-power-free}, we obtain the following.

\begin{proposition}\label{cor:phi-s-5-4-power-free}
The infinite word $\varphi(\s)$ is $5/4$-power-free.
\end{proposition}

\begin{proof}
By Theorem~\ref{thm:s-pre-5-4-power-free}, $\s$ is pre-$5/4$-power-free.
Since $\s$ is a word on $\Sigma_8\setminus\Gamma$, Proposition~\ref{pro:pre-5-4-power-phi} implies that $\varphi(\s)$ is also pre-$5/4$-power-free.
Proposition~\ref{pro:pre-5-4-power-strong-cond} implies that $\varphi(\s)$ is $5/4$-power-free.
\end{proof}

The next lemma shows that applying the coding $\tau$ to $\varphi(\s)$ preserves $5/4$-power-freeness.

\begin{lemma}\label{lem:tau-phi-s-5-4-power-free}
The infinite word $\tau(\varphi(\s))$ is $5/4$-power-free.
\end{lemma}

\begin{proof}
Let $xyx'$ be a nonempty factor of $\tau(\varphi(\s))$ with $\size{x} = \frac{1}{3} \size{y} = \size{x'}$.
We show that $x \neq x'$.
Let $uvu'$ be the factor of $\varphi(\s)$ corresponding to $xyx'$ in $\tau(\varphi(\s))$, where $\tau(u)=x$, and $\tau(v)=y$, and $\tau(u')=x'$.
Let $\ell = \size{x}$.

First assume $\ell$ is even.
Assume toward a contradiction that $x = x'$.
Since $\size{u v} = 4\ell$ is a multiple of $8$, the sequences of subscripts in $u$ and $u'$ line up.
Then $\tau(u)=x=\tau(u')$ implies that $u=u'$, and $uvu$ is a $5/4$-power in $\varphi(\s)$.
This violates Proposition~\ref{cor:phi-s-5-4-power-free}.

Assume that $\ell$ is odd.
Then $4\ell \equiv 4 \mod{8}$, and the sequences of subscripts in $u$ and $u'$ differ by $4$.
For convenience, we give a table of the values of $\tau(\varphi(n_j))$:
\begin{equation}\label{eqn:array-phi}
	\begin{array}{clllllll}
		\tau(\varphi(n_0)) & = & 0 & 1 & 0 & 0 & 1 & (n + 3) \\
		\tau(\varphi(n_1)) & = & 1 & 1 & 0 & 0 & 0 & (n + 2) \\
		\tau(\varphi(n_2)) & = & 1 & 1 & 1 & 0 & 0 & (n + 3) \\
		\tau(\varphi(n_3)) & = & 0 & 1 & 1 & 0 & 1 & (n + 2) \\
		\tau(\varphi(n_4)) & = & 0 & 1 & 0 & 0 & 1 & (n + 1) \\
		\tau(\varphi(n_5)) & = & 1 & 1 & 0 & 0 & 0 & (n + 2) \\
		\tau(\varphi(n_6)) & = & 1 & 1 & 1 & 0 & 0 & (n + 1) \\
		\tau(\varphi(n_7)) & = & 0 & 1 & 1 & 0 & 1 & (n + 2).
	\end{array}
\end{equation}
Recall that $\s$ is a word on $\Sigma_8\setminus\Gamma$ by Part~\ref{stuff about s} of Lemma~\ref{lem:various-properties}.

If $\ell=1$, it is sufficient to check that the letter in position $i$ in $\tau(\varphi(\s))$ is different from the letter in position $i+4$ since Table~\eqref{eqn:array-phi} lists the letters in images of $\tau\circ\varphi$.
We do this by looking at six pairs of columns, each separated by $3$ columns, in Table~\eqref{eqn:array-phi}.
The first and fifth columns are unequal row by row.
For the second and sixth columns, several potential problems occur.
For instance, in the fifth row, $\tau(\varphi(0_4))=0 1 0 0 1 1$ contains the $5/4$-power $10011$, but fortunately $0_4$ never appears in $\s$.
Similarly, $-2_0, -1_1, -2_2, -1_3, -1_5, 0_6, -1_7$ create $5/4$-powers but never appear in $\s$.
For the third and first columns, we have to compare letters that are offset by $1$ row.
In particular, the last letter of the third column has to be compared with the first letter in the first column.
We do the same for the remaining three pairs of columns and find that there are no $5/4$-powers of length $5$ in $\tau(\varphi(\s))$.

Suppose that $\ell\ge 3$.
Each occurrence of $u$ in $\varphi(\s)$ contains a letter with an even subscript, that is, a letter that falls either in the first, third, or fifth column in Table~\eqref{eqn:array-phi}.
Since the subscripts in $u$ and $u'$ differ by $4$, the pair of corresponding letters in $u$ and $u'$ belongs to
\[
\{(0_0,1_4),(0_2,1_6),(1_4,0_0),(1_6,0_2)\}
\]
by Part~\ref{stuff about s} of Lemma~\ref{lem:various-properties}, as in the proof of Proposition~\ref{pro:pre-5-4-power-phi}.
This shows that $\tau(u)\neq \tau(u')$, which implies $x \neq x'$.
\end{proof}

\begin{remark}
Note that the previous argument works more generally to show that if $w$ is a subscript-increasing pre-$5/4$-power-free word on $\Sigma_8\setminus\Gamma$ then $\tau(\varphi(w))$ is $5/4$-power-free.
\end{remark}

The last step in showing that $\p\tau(\varphi(\s))$ is $5/4$-power-free is to prove that prepending $\p$ to $\tau(\varphi(\s))$ also yields a $5/4$-power-free word.
To that aim, we introduce the following notion.

\begin{definition}
Let $N$ be a set of integers, and let $\alpha, \beta \in \Z \cup \{n+1,n+2,n+3\}$, where $n$ is a symbol.
Then $\alpha$ and $\beta$ are \emph{possibly equal with respect to $N$} if there exist $m, m' \in N$ such that $\alpha\vert_{n=m} = \beta\vert_{n=m'}$.
\end{definition}

Two letters $\alpha,\beta$ are possibly equal with respect to $N$ if we can make them equal by substituting integers from $N$ for the symbol $n$.
In particular, for every nonempty set $N$, two integers $\alpha, \beta$ are possibly equal if and only if $\alpha = \beta$.
The definition of possibly equal letters extends to words on $\Z \cup \{n+1,n+2,n+3\}$ in the natural way.
The next two lemmas will be used to prove Theorem~\ref{thm:p-tau-phi-s-5-4-power-free}.

\begin{lemma}\label{lem:weird technical 2}
Let $n$ be a symbol, and let $N \supseteq \{-3, -2, \dots, 4\}$.
Let $\alpha, \beta$ be elements of $\{0, 1, \dots, 5\} \cup \{n + 1, n + 2, n + 3\}$.
If $\alpha$ and $\beta$ are possibly equal with respect to $N$, then they are possibly equal with respect to $\{-3, -2, \dots, 4\}$.
\end{lemma}

\begin{proof}
Suppose $\alpha$ and $\beta$ are possibly equal with respect to $N$.
There are three cases to consider depending on the nature of the letters $\alpha$ and $\beta$.

If both letters are integers, then $\alpha = \beta$.
It follows that $\alpha$ and $\beta$ are possibly equal with respect to $\{-3, -2, \dots, 4\}$.

If one letter is an integer and the other is symbolic, without loss of generality let $\alpha \in \{0, 1, \dots, 5\}$ and $\beta = n+d$ for some $d\in \{1,2,3\}$.
By assumption, $\alpha=n+d$ for some $n\in N$, namely $n = \alpha - d \in \{0, 1, \dots, 5\} - \{1, 2, 3\} = \{-3, -2, \dots, 4\}$.
Therefore $\alpha$ and $\beta$ are possibly equal with respect to $\{-3, -2, \dots, 4\}$.

If both letters are symbolic, write $\alpha = n+d$ and $\beta = n+d'$ with $d,d'\in \{1,2,3\}$.
Without loss of generality, $d \leq d'$.
Let $m, m' \in N$ such that $m + d = m' + d'$.
Then $m - m' = d' - d \in \{0, 1, 2\}$.
Then $\alpha$ and $\beta$ are possibly equal with respect to $\{-3, -2, \dots, 4\}$, since $(m - m') + d = 0 + d'$ and $m - m', 0 \in \{-3, -2, \dots, 4\}$.
\end{proof}

\begin{lemma}\label{lem:weird technical}
Let $n$ be a symbol, and let $N = \{-3, -2, \dots, 4\}$.
For all $\alpha \in\Z \cup \{n+1,n+2,n+3\}$, define the set
\[
X_\alpha
=
\begin{cases}
\{\alpha\} & \text{if $\alpha \in \Z$} \\
N+d & \text{if $\alpha \in \{n+1,n+2,n+3\}$}.
\end{cases}
\]
If $\alpha,\beta \in\{0, 1, \dots, 5\} \cup \{n+1,n+2,n+3\}$ are possibly equal with respect to $N$, then $X_\alpha \cap X_\beta$ is nonempty.
\end{lemma}

\begin{proof}
The set $X_\alpha$ is the set of integer letters $c$ such that the letter $\alpha$ is possibly equal to $c$ with respect to $N$.
Suppose the letters $\alpha,\beta \in\{0, 1, \dots, 5\} \cup \{n+1,n+2,n+3\}$ are possibly equal with respect to $N$.
There are three cases to consider depending on the nature of these letters.

If both letters are integers, then $\alpha=\beta$, so $\alpha \in X_\alpha \cap X_\beta$.

If one letter is an integer and the other is symbolic, without loss of generality, let $\alpha\in\{0, 1, \dots, 5\}$ and $\beta=n+d$ with $d\in \{1,2,3\}$.
By assumption, $\alpha=n+d$ for some $n\in N$.
So $X_\alpha \cap X_\beta = \{\alpha\} \cap (N+d) = \{\alpha\}$.

If both letters are symbolic, let $\alpha=n+d$ and $\beta=n+d'$ with $d,d'\in \{1,2,3\}$.
Then $X_\alpha = N +d$ and $X_\beta = N + d'$, so $0 \in X_\alpha \cap X_\alpha$.
\end{proof}

\begin{theorem}\label{thm:p-tau-phi-s-5-4-power-free}
The infinite word $\p\tau(\varphi(\s))$ is $5/4$-power-free.
\end{theorem}

\begin{proof}
Since $\p$ is the prefix of $\word_{5/4}$ of length $6764$, $\p$ is $5/4$-power-free.
By Lemma~\ref{lem:tau-phi-s-5-4-power-free}, $\tau(\varphi(\s))$ is also $5/4$-power-free.
So if $\p\tau(\varphi(\s))$ contains a $5/4$-power, then it must overlap $\p$ and $\tau(\varphi(\s))$.
We will show that there are no $5/4$-powers $xyx$ starting in $\p$.

For factors $xyx$ with $\size{x} < 952$ starting in $\p$, note that $\lvert x \rvert < 952$ implies $\lvert xyx \rvert < 5 \cdot 952$, so it is enough to look for $5/4$-powers in $\p\tau(\varphi(\z))$ --- as opposed to $\p\tau(\varphi(\s))$ --- starting in $\p$.
We check programmatically that there is no such $5/4$-power $xyx$.
The computation took about a minute.

For longer factors, we show that each length-$952$ factor $x$ starting in $\p$ only occurs once in $\p\tau(\varphi(\s))$.
This will imply that there is no $5/4$-power $xyx$ in $\p\tau(\varphi(\s))$ starting in $\p$ such that $\lvert x \rvert \ge 952$.
Since $\s=0_2 0_3 3_4 \cdots$, the word $\p\tau(\varphi(\s))$ is of the form
\[
\p \tau \Big( \varphi(n_2) \varphi(n_3) \varphi(n_4) \varphi(n_5) \varphi(n_6) \varphi(n_7) \varphi(n_0) \varphi(n_1) \cdots \Big).
\]
Here we abuse notation; namely, the $n$'s are not necessarily equal.
Observe that $\size{\varphi(n_2) \varphi(n_3) \cdots \varphi(n_0) \varphi(n_1)}=48$.
We use a method to distinguish factors based on~\cite[Section 6]{Pudwell--Rowland}, where one would consider the set $\{0,1,\dots, \size{\p } + 48- 1\}$ of initial positions.
However, the morphism $\varphi$ makes things more complicated.
We need to run the procedure in the following paragraph for two different sets of positions instead of one.
Indeed, although the sequence $3, 2, 1, 2, 1, 2, 3, 2, \dots$ of increments $d$ has period length $8$, each of the first five columns has period length at most $4$.
This implies that the factors starting at positions $i$ and $i+24$ are possibly equal for all $i$ sufficiently large, so one set of positions does not suffice.
The two sets are
\[
S_1 =\{0,1,\dots, \size{\p } - 1\} \cup \{\size{\p }, \size{\p }+1, \dots, \size{\p } + 23\}
\]
and
\[
S_2 =\{0,1,\dots, \size{\p } - 1\} \cup \{\size{\p } + 24, \size{\p } + 25,\dots, \size{\p } + 47\}.
\]
The positions in $\{0,1,\dots, \size{\p } - 1\}$ represent factors starting in the prefix $\p$ of $\p\tau(\varphi(\s))$, while the other positions are representatives of general positions modulo $48$ in the suffix $\tau(\varphi(\s))$.
We also need to specify a set $N$ of integers that, roughly speaking, represent the possible values that each symbolic $n$ can take.

Let $S$ be a set of positions, and let $N$ be a set of integers.
As in Lemma~\ref{lem:weird technical}, the set $N$ represents values of $n$ such that the last letter of $\tau(\varphi(n_j))$, namely $n + d$ for some $d \in \{1, 2, 3\}$, is a letter in $\p\tau(\varphi(\s))$.
We maintain classes of positions corresponding to possibly equal factors starting at those positions.
Start with $\ell=0$, for which all length-$0$ factors are equal.
Then all positions belong to the same class $S$.
At each step, we increase $\ell$ by $1$, and for each position $i$ we consider the factor of length $\ell$ starting at position $i$, extended from the previous step by one letter to the right.
We break each class into new classes according to the last letter of each extended factor, as described in the following paragraph.
We stop once each class contains exactly one position, because then each factor occurs at most once.
Note that this procedure does not necessarily terminate, depending on the inputs.
If it terminates, then we record $\ell$.

For each class $\mathcal{I}$, we build subclasses $\mathcal{I}_c$ indexed by integers $c$.
For each position $i \in \mathcal{I}$, we consider the extended factor of length $\ell$ starting at position $i$.
If $0\le i\le \size{\p}-1$, then the new letter is in $\Z$ because $i + \ell - 1$ represents a particular position in $\p\tau(\varphi(\s))$.
If $i\ge \size{\p}$, then the new letter is either in $\Z$ or symbolic in $n$ because $i + \ell - 1$ represents all sufficiently large positions congruent to $i + \ell - 1$ modulo $48$.
Now there are two cases.
If the new letter is an integer $c$, we add the position $i$ to the class $\mathcal{I}_c$.
If the new letter is $n+d$ where $d\in\{1,2,3\}$, then we add the position $i$ to the class $\mathcal{I}_{n' + d}$ for each $n' \in N$.
We do this for all classes $\mathcal{I}$ and we use the union
\[
\bigcup_{\mathcal{I}} \, \{ \mathcal{I}_c : c \in \Z \}
\]
as our new set of classes for the next length $\ell + 1$.

For the sets $S_1$ and $S_2$, we initially use $N = \{0, 1, 2, 3, 4\}$.
For both sets, this procedure terminates and gives $\ell = 952$.
Our implementation took about $10$ seconds each.
The prefix of $\p\tau(\varphi(\s))$ of length $\size{\p} + 952-1$ is a word on the alphabet $\{0, 1, \dots, 5\}$.
Therefore, since $d \in \{1, 2, 3\}$, at most the eight classes $\mathcal{I}_0, \mathcal{I}_1, \dots, \mathcal{I}_7$ arise in each step of the procedure, since $\{0, 1, \dots, 5\} \cup (N + \{1, 2, 3\}) = \{0, 1, \dots, 7\}$.

It remains to show that using the set $N = \{0, 1, 2, 3, 4\}$ is sufficient to guarantee that, since the procedures terminated, each length-$952$ factor $x$ starting in $\p$ only occurs once in $\p\tau(\varphi(\s))$.
Since $\p\tau(\varphi(\s))$ is a word on the alphabet $\N$, it suffices to choose a subset $N$ of $\N - \{1, 2, 3\} = \{-3, -2, \dots\}$.
There exist letters in $\p\tau(\varphi(\s))$ that arise as the last letter of $\tau(\varphi(n_j))$ for arbitrarily large $n$, but the procedure cannot use an infinite set $N$.
We use Lemmas~\ref{lem:weird technical 2} and \ref{lem:weird technical} to show that $N$ need not contain any integer greater than $4$.

The procedure examines factors of a prefix of $\p\tau(\varphi(\s))$ and
\[
	\tau \Big( \varphi(n_2) \varphi(n_3) \varphi(n_4) \varphi(n_5) \varphi(n_6) \varphi(n_7) \varphi(n_0) \varphi(n_1) \cdots \Big).
\]
The factors of $\p\tau(\varphi(\s))$ lie in a prefix of length at most $\size{\p} + 952 - 1$.
The other word $\tau ( \varphi(n_2) \varphi(n_3) \varphi(n_4) \cdots )$ is on $\{0, 1\} \cup \{n + 1, n + 2, n + 3\}$.
Both words are on the alphabet $\{0, 1, \dots, 5\} \cup \{n + 1, n + 2, n + 3\}$, so we will be able to apply Lemma~\ref{lem:weird technical}.
Let $N \supseteq \{0, 1, \dots, 5\} - \{1, 2, 3\} = \{-3, -2, \dots, 4\}$.
On the step corresponding to length $\ell$ in the procedure, suppose the length-$\ell$ factors starting at positions $i$ and $j$ are possibly equal with respect to $N$.
We will show that there is a class $\mathcal{I}_c$ containing $i$ and $j$.
By Lemma~\ref{lem:weird technical 2}, the two factors are possibly equal with respect to $\{-3, -2, \dots, 4\}$.
In particular, the last two letters, which have positions $i + \ell - 1$ and $j + \ell - 1$, are possibly equal with respect to $\{-3, -2, \dots, 4\}$.
Let $\alpha$ and $\beta$ be these two letters.
Recall that $\alpha, \beta\in\{0, 1, \dots, 5\} \cup \{n + 1, n + 2, n + 3\}$.
By Lemma~\ref{lem:weird technical}, there exists $c \in X_{\alpha} \cap X_{\beta}$.
Therefore the letters $\alpha$ and $\beta$ are both possibly equal to $c$ with respect to $\{-3, -2, \dots, 4\}$, since $X_\alpha$ is the set of integers possibly equal to $\alpha$ with respect to $\{-3, -2, \dots, 4\}$.
So $i$ and $j$ are both added to $\mathcal I_c$.
We have shown that $N \subseteq \{-3, -2, \dots, 4\}$ suffices.

Next we remove $-3$ and $-2$.
We continue to consider letters in $\p\tau(\varphi(\s))$ that arise as the last letter of $\tau(\varphi(n_j))$ for $n\in N$.
Since $\tau(\s)$ does not contain the letter $-3$, this implies $n + 3$ and $0$ are not possibly equal with respect to the alphabet of $\tau(\s)$, so $N$ need not contain $-3$.
Similarly, $\tau(\s)$ does not contain the letter $-2$; therefore $n + 2$ and $0$ are not possibly equal, and $n + 3$ and $1$ are not possibly equal, so $N$ need not contain $-2$.
Therefore $N \subseteq \{-1, 0, 1, 2, 3, 4\}$ suffices.

To remove $-1$, we run the procedure on the sets $S_1$ and $S_2$ again.
However, this time we use the set $\{-1, 0, 1, 2, 3, 4\}$ and we artificially stop the procedure at $\ell = 952$.

For the set $S_1$, stopping at $\ell = 952$ yields $4$ nonempty classes of positions remaining, namely
\[
	\{\{6760, 6784\}, \{6761, 6785\}, \{6762, 6786\}, \{6763, 6787\}\}.
\]
The smallest position in each class is one of the last $4$ positions in $\p$.
As length-$952$ factors of $\p \tau ( \varphi(n_2) \varphi(n_3) \cdots)$, each pair of factors starting at those positions are possibly equal with respect to $N$.
For instance, consider the two factors
\begin{align*}
	& 0 0 0 \phantom{(n+{}}3\phantom{)} 1 1 1 0 0 \phantom{(n+{}}3\phantom{)} 0 1 1 0 1 \phantom{(n+{}}2\phantom{)} 0 1 0 0 1 \phantom{(n+{}}4\phantom{)} 1 1 0 0 0 \phantom{(n+{}}2\phantom{)} 1 1 1 0 0\phantom{(n+{}}2\phantom{)} \cdots, \\
	& 0 0 0 (n+2) 1 1 1 0 0 (n+1) 0 1 1 0 1 (n+2) 0 1 0 0 1 (n+3) 1 1 0 0 0 (n+2) 1 1 1 0 0 (n+3) \cdots
\end{align*}
starting at positions $6760, 6784$.
The first factor is a prefix of $0003\tau(\varphi(\s))$ and the second is a prefix of $0 0 0 (n+2)\tau(\varphi(n_6)\varphi(n_7)\cdots)$, which occurs every $48$ positions in the periodic word $\tau ( \varphi(n_2) \varphi(n_3) \cdots)$.
For these two factors to be equal, the pair of letters $2$ and $n+3$ have to be equal, and solving $2 = n + 3$ gives $n=-1$.
Similarly, the other three pairs of factors are only equal if the same pair of letters are equal, which again gives $n = -1$.
But letters $-1$ only appear in $\s$ in its prefix $\z$ and only with subscripts $0$ and $2$ and only in the nine positions $6, 14, 16, 32, 40, 48, 56, 70, 80$.
A finite check shows that the factors in each pair are different.

For $S_2$, there are also $4$ nonempty classes of positions remaining:
\[
	\{\{6760, 6808\}, \{6761, 6809\}, \{6762, 6810\}, \{6763, 6811\}\}.
\]
To show that the factors in each pair are different, we use slightly longer prefixes of $003\tau(\varphi(\s))$ and $0 0 0 (n+2)\tau(\varphi(n_2)\varphi(n_3)\cdots)$ than we used for $S_1$, and we again find a pair of letters $2$ and $n+3$.
This again implies $n=-1$ for each class.

Therefore we can remove $-1$ from $N$.
So $N = \{0, 1, 2, 3, 4\}$ suffices.
\end{proof}

\section{Lexicographic-leastness}\label{sec:lexicographic-least}

In this section we show that $\p\tau(\varphi(\s))$ is lexicographically least by showing the following.

\begin{theorem}\label{thm:5over4-lex-least}
Decreasing any nonzero letter of $\p\tau(\varphi(\s))$ introduces a $5/4$-power ending at that position.
\end{theorem}

\begin{proof}
We proceed by induction on the positions $i$ of letters in $\p\tau(\varphi(\s))$.
As a base case, since we have computed a long enough common prefix of $\word_{5/4}$ and $\p\tau(\varphi(\s))$, decreasing any nonzero letter of $\p\tau(\varphi(\s))$ in position $i\in\{0,1, \dots, 331 039\}$ introduces a $5/4$-power in $\p\tau(\varphi(\s))$ ending at that position.

Now suppose that $i\ge 331 040 = 31\size{\p} + 6\size{\z}$ and assume that decreasing a nonzero letter in any position less than $i$ in $\p\tau(\varphi(\s))$ introduces a $5/4$-power in $\p\tau(\varphi(\s))$ ending at that position.
We will show that decreasing the letter in position $i$ in $\p\tau(\varphi(\s))$ introduces a $5/4$-power in $\p\tau(\varphi(\s))$ ending at that position.
Since $\size{\p}=6764$ and $\size{\z}=20226$, observe that this letter in position $i \geq 128120 = \size{\p} + 6 \size{\z} = \size{\p\tau(\varphi(\z))}$ actually belongs to the suffix $\tau(\varphi^2(\s))$, and its position in $\tau(\varphi^2(\s))$ is $i-\size{\p}-6\size{\z}$.
Every such letter is a factor of $\tau(\varphi(n_j))$ for some $n\in \N$ and some $j\in\{0,1,\dots, 7\}$.
We make use of the array~\eqref{eqn:array-phi} of letters of $\varphi$.

If $i-\size{\p}-6\size{\z} \nequiv 5 \mod 6$, then the letter in position $i$ belongs to one of the first five columns.
Any $0$ letters cannot be decreased.
Observe that the fourth column is made of letters $0$.
Since each letter in the second column is $1$, decreasing any letter $1$ to $0$ in the second column produces a new $5/4$-power of length $5$ of the form $0y0$ between the fourth and second columns.
Since the even-subscript letters in $\varphi(\s)$ form the word $(1_41_60_00_2)^\omega$ by the proof of Part~\ref{stuff about s} in Lemma~\ref{lem:various-properties}, then decreasing any letter $1$ to $0$ in the first, third, or fifth column introduces a $5/4$-power of length $5$.

Otherwise $i-\size{\p}-6\size{\z} \equiv 5 \mod{6}$, that is, the letter in position $i$ is in the sixth column.
These letters arise as $n + d$ for some $n \in \N$ and $d \in \{1, 2, 3\}$.
By Parts~\ref{stuff about s} and \ref{last-letter-phi} of Lemma~\ref{lem:various-properties}, $n + d \geq 2$.
If we decrease $n + d$ to $0$, then we create one of the following $5/4$-powers of length $10$:
\begin{align*}
	& 10 \cdot 1(n+2) 0100 \cdot 10 \\
	& 00 \cdot 1(n+3) 1100 \cdot 00 \\
	& 00 \cdot 0(n+2) 1110 \cdot 00 \\
	& 10 \cdot 0(n+3) 0110 \cdot 10 \\
	& 10 \cdot 1(n+2) 0100 \cdot 10 \\
	& 00 \cdot 1(n+1) 1100 \cdot 00 \\
	& 00 \cdot 0(n+2) 1110 \cdot 00 \\
	& 10 \cdot 0(n+1) 0110 \cdot 10.
\end{align*}
If we decrease $n + d$ to $1$, then we create a new $5/4$-power of length $5$ because each letter in the second column is $1$.

It remains to show that decreasing the letter $n + d$ in position $i$ in $\tau(\varphi^2(\s))$ to a letter $c \in \{2, 3, \dots, n + d - 1\}$ introduces a $5/4$-power ending at that position.
Intuitively, this operation corresponds to decreasing a letter $n_j$ in the preimage $\varphi(\s)$ to $(c-d)_j$ for some $0\le j \le 7$.
In particular, the last letter of $\tau(\varphi((c-d)_j))$ is $c$.
We examine three cases according to the value of $d$.

\textbf{Case 1.}
If $d=1$, then Notation~\ref{not:phi} implies that the corresponding letter in position $i$ in $\varphi^2(\s)$ is $(n+1)_5$ or $(n+1)_1$.
We see that $(n+1)_5$ and $(n+1)_1$ appear in the images of $n_4$ and $n_6$ under $\varphi$.
By Parts~\ref{stuff about s} and \ref{even-subscript} of Lemma~\ref{lem:various-properties}, the only letters with subscripts $4$ and $6$ in $\varphi(\s)$ are $1_4$ and $1_6$.
Therefore $n = 1$, and there is nothing to check since $\{2, 3, \dots, n + d - 1\}$ is the empty set.

\textbf{Case 3.}
If $d=3$, then the corresponding letter in position $i$ in $\varphi^2(\s)$ is $(n+3)_5$ or $(n+3)_1$, which appear in the images of $n_0$ and $n_2$ under $\varphi$.
By Parts~\ref{stuff about s} and \ref{even-subscript} of Lemma~\ref{lem:various-properties}, the only letters with subscripts $0$ and $2$ in $\varphi(\s)$ are $0_0$ and $0_2$, so $n = 0$ and $c \in \{2, 3, \dots, n + d - 1\} = \{2\}$.
By Part~\ref{stuff about s} of Lemma~\ref{lem:various-properties}, the letter four positions before $0_0$ is $1_4$, and the letter four positions before $0_2$ is $1_6$.
Therefore decreasing $n+3$ in $3_5$ or $3_1$ to $c=2$ introduces one of following $5/4$-powers of length $30$:
\begin{align*}
	\tau(\varphi(1_4 \bar{n}_5 \bar{\bar{n}}_6 \bar{\bar{\bar{n}}}_7)) 010012
	&= 010012 \tau(\varphi(\bar{n}_5 \bar{\bar{n}}_6 \bar{\bar{\bar{n}}}_7)) 010012 \\
	\tau(\varphi(1_6 \bar{n}_7 \bar{\bar{n}}_0 \bar{\bar{\bar{n}}}_1)) 111002
	&= 111002 \tau(\varphi(\bar{n}_7 \bar{\bar{n}}_0 \bar{\bar{\bar{n}}}_1)) 111002.
\end{align*}

\textbf{Case 2.}
If $d=2$, then the corresponding letter in position $i$ in $\varphi^2(\s)$ is $(n+2)_3$ or $(n+2)_7$, which appear in the images of $n_1$, $n_3$, $n_5$, and $n_7$ under $\varphi$.
Let $w$ be the length-$(i-\size{\p}-6\size{\z}+1)$ prefix of $\tau(\varphi^2(\s))$ with last letter $n+2$.
Since $i-\size{\p}-6\size{\z}\equiv 5 \mod{6}$, let $u$ be the prefix of $\varphi(\s)$ of length $\frac{i-\size{\p}-6\size{\z}+1}{6}$.
Then $\tau(\varphi(u))=w$ and $\tau(u)$ ends with $n$ as pictured.
\begin{center}
\scalebox{0.6}{
\begin{tikzpicture}
\draw[] (0,0) -- (19,0);
\draw[] (0,1) -- (19,1);
\draw[] (0,0) -- (0,1);
\draw[dashed] (19,0) -- (20,0);
\draw[dashed] (19,1) -- (20,1);

\node[] at (1.5,0.5) {$\p$};
\draw[] (3,0) -- (3,1);
\node[] at (6.5,0.5) {$\tau(\varphi(\z))$};
\draw[] (10,0) -- (10,1);
\node[] at (11.5,0.5) {$\tau(\varphi^2(\z)\varphi^3(\z)\cdots)$};

\draw[] (10,-1.3) -- (10,-0.7);
\draw[] (10,-1) -- (18.4,-1);
\draw[] (18.4,-1.3) -- (18.4,-0.7);
\node[] at (14.5,-0.7) {$w$};

\draw[] (18.4,1.3) -- (18.4,1);
\node[] at (18.4,1.7) {$i$};

\draw[] (3,-1.3) -- (3,-0.7);
\draw[] (3,-1) -- (4.4,-1);
\draw[] (4.4,-1.3) -- (4.4,-0.7);
\node[] at (3.7,-0.7) {$\tau(u)$};
\end{tikzpicture}
}
\end{center}
Let $w'$ be the word obtained by decreasing the last letter $n+2$ in $w$ to $c$, and let $u'$ be the word obtained by decreasing the last letter $n_j$ in $u$ to $(c-2)_j$.
Then $\tau(\varphi(u'))=w'$.
By the induction hypothesis, $\p \tau(u')$ contains a $5/4$-power suffix $xyx$.
Now we consider two subcases, depending on where $xyx$ starts in $\p \tau(u')$ as depicted below.
(Case~2.2 contains two possibilities, but we show that the first does not actually occur.)
\begin{center}
\scalebox{0.6}{
\begin{tikzpicture}

\draw[dashed] (5,-3) -- (5,-6.5);

\draw[] (0,-3) -- (19,-3);
\draw[] (0,-2) -- (19,-2);
\draw[] (0,-3) -- (0,-2);
\draw[] (19,-3) -- (19,-2);

\node[] at (2.5,-2.5) {$\p$};
\draw[] (5,-3) -- (5,-2);
\node[] at (11.5,-2.5) {$\tau(u')$};

\draw[] (9,-3.75) -- (19,-3.75);
\draw[] (9,-4.05) -- (9,-3.45);
\node[] at (10,-3.45) {$x$};
\draw[] (11,-4.05) -- (11,-3.45);
\node[] at (14,-3.45) {$y$};
\draw[] (17,-4.05) -- (17,-3.45);
\node[] at (18,-3.45) {$x$};
\draw[] (19,-4.05) -- (19,-3.45);

\node[] at (0.5,-3.75) {Case 2.1};
\draw[loosely dashed] (0,-4.4) -- (19,-4.4);
\node[] at (0.5,-5.6) {Case 2.2};

\draw[] (0.25,-5) -- (19,-5);
\draw[] (0.25,-5.3) -- (0.25,-4.7);
\node[] at (2.125,-4.7) {$x$};
\draw[] (4,-5.3) -- (4,-4.7);
\node[] at (9.625,-4.7) {$y$};
\draw[] (15.25,-5.3) -- (15.25,-4.7);
\node[] at (17.125,-4.7) {$x$};
\draw[] (19,-5.3) -- (19,-4.7);

\draw[] (4,-6) -- (19,-6);
\draw[] (4,-6.3) -- (4,-5.7);
\node[] at (5.5,-5.7) {$x$};
\draw[] (7,-6.3) -- (7,-5.7);
\node[] at (12,-5.7) {$y$};
\draw[] (16,-6.3) -- (16,-5.7);
\node[] at (17.5,-5.7) {$x$};
\draw[] (19,-6.3) -- (19,-5.7);
\end{tikzpicture}
}
\end{center}

Case 2.1.
Suppose that $xyx$ starts after $\p$ in $\p \tau(u')$, that is, $xyx$ is a suffix of $\tau(u')$.
Write $x=\tau(x')=\tau(x'')$ and $y=\tau(y')$ where $x'y'x''$ is the corresponding subscript-increasing suffix of $u'$.
Since $\size{xy}$ is divisible by $4$, the subscripts in $x'$ and $x''$ are either equal or differ by $4$.
Since $d=2$, the subscripts of the last letters of $x'$ and $x''$ are odd.
If $\size{x}=1$, then $x=c-2$ and $\varphi(x')=\varphi(x'')$ by definition of $\varphi$.
Now if $\size{x}\ge 2$, the subscripts of the penultimate letters of $x'$ and $x''$ are even and either equal each other or differ by $4$.
Since the subscripts cannot differ by $4$ by Part~\ref{stuff about s} of Lemma~\ref{lem:various-properties}, they must be equal.
Since $\tau(x')=\tau(x'')$, we have $x'=x''$, so $\varphi(x')=\varphi(x'')$.
Then $\p \tau(\varphi(\z)) w'=\p \tau(\varphi(\z u'))$ contains the $5/4$-power $\tau(\varphi(x'y'x''))$ as a suffix.
Therefore, decreasing $n+2$ to $c$ introduces a $5/4$-power in $\p\tau(\varphi(\s))$ ending at position $i$.

Case 2.2.
Suppose that $xyx$ starts before $\tau(u')$ in $\p \tau(u')$.
In particular, $\tau(u')$ is a suffix of $xyx$.
Since $i\ge 331 040 = 31\size{\p} + 6\size{\z}$,
\[
 5\size{\p} < \frac{i-\size{\p}-6\size{\z}+1}{6} = \size{u} = \size{\tau(u')} \le \size{xyx} = 5\size{x},
 \]
which implies $\size{\p} < \size{x}$.
Therefore the first $x$ overlaps $\p$ but is not a factor of $\p$.
Suppose the overlap length is at least $5$.
Then the first $x$ contains $20003$ as a factor.
But since $20003$ is never a factor of an image under $\tau \circ \varphi$, so the overlap length of $x$ and $\p$ is at most $4$.
Then the first $x=sv$ is made of a nonempty suffix $s$ of $0003$ followed by a nonempty prefix $v$ of $\tau(u')$ such that $vyx=\tau(u')$.
Write $v=\tau(v')$, $x=\tau(x'')$, and $y=\tau(y')$ where $v'y'x''=u'$.
To get around the fact that $\p$ does not have subscripts, we use $\z$ instead to obtain a preimage of $s$ under $\tau$.
Recall that $0003$ is a common suffix of $\p$ and $\tau(\z)$, and the corresponding suffix in $\z$ is $0_0 0_1 0_2 3_3$.
So let $s'$ be the suffix of $0_0 0_1 0_2 3_3$ such that $s=\tau(s')$.
Now observe that $x'=s'v'$ is a subscript-increasing factor of $\z u'$, overlapping $\z$.
Thus $\varphi(x'y'x'')$ is a subscript-increasing suffix of $\varphi(\z u')$, overlapping $\varphi(\z)$.
By Part~\ref{stuff about s} of Lemma~\ref{lem:various-properties} again, since $\tau(x')=\tau(s'v')=sv=x=\tau(x'')$ and $\size{x}\ge 2$, the subscripts in $x'$ and $x''$ are equal.
Consequently, $x'=x''$ and $\varphi(x')=\varphi(x'')$.
Then $\p \tau(\varphi(\z)) w'=\p \tau(\varphi(\z u'))$ contains the $5/4$-power $\tau(\varphi(x'y'x''))$ as a suffix.
Therefore, decreasing $n+2$ to $c$ introduces a $5/4$-power in $\p\tau(\varphi(\s))$ ending at position $i$.
\end{proof}

Theorems~\ref{thm:p-tau-phi-s-5-4-power-free} and \ref{thm:5over4-lex-least} imply Theorem~\ref{thm:main}, which states that $\word_{5/4} = \p \tau(\varphi(\s))$.

\section{The sequence of letters in $\word_{5/4}$}\label{sec:6-regularity}

In this section we prove Corollary~\ref{cor:5over4}, which states that the sequence $w(i)_{i \geq 0}$ of letters in $\word_{5/4}$ satisfies
\begin{equation}\label{eq:recurrence-6th-column}
	w(6 i + 123061) =
	w(i + 5920) + \begin{cases}
		3	& \text{if $i \equiv 0, 2 \mod 8$} \\
		1	& \text{if $i \equiv 4, 6 \mod 8$} \\
		2	& \text{if $i \equiv 1 \mod 2$}
	\end{cases}
\end{equation}
for all $i\ge 0$.
Then we prove Theorem~\ref{thm:regularity}, which states that a sequence satisfying a certain recurrence is $k$-regular.
In particular, the sequence of letters in $\word_{5/4}$ is $6$-regular.

\begin{proof}[Proof of Corollary~\ref{cor:5over4}]
Since $\size{\p}=6764$ and $\size{\p\tau ( \varphi(\z))}=128120$, we have
\[
w(6764) w(6765) \cdots = \tau (\varphi(\s))
\]
and
\[
w(128120) w(128121) \cdots = \tau (\varphi^2(\s))
\]
by Theorem~\ref{thm:main}.
The first letter in $\varphi(\s)$ is $1_4$, so the definition of $\varphi$ implies
\[
	w(6 i + 128125) =
	w(i + 6764) + \begin{cases}
		1	& \text{if $i \equiv 0, 2 \mod 8$} \\
		3	& \text{if $i \equiv 4, 6 \mod 8$} \\
		2	& \text{if $i \equiv 1 \mod 2$}
	\end{cases}
\]
for all $i\ge 0$.
Thus Equation~\eqref{eq:recurrence-6th-column} holds for $i\ge 844$.
Note that, since $844 \equiv 4 \mod 8$, the cases $i \equiv 0, 2 \mod 8$ and $i \equiv 4, 6 \mod 8$ are switched relative to Equation~\eqref{eq:recurrence-6th-column}.
Finally, we check programmatically that Equation~\eqref{eq:recurrence-6th-column} holds for all $0\le i \le 843$.
Alternatively, this follows from the fact that the length-$844$ suffixes of $\p$ and $\tau(\z)$ are equal by Part~\ref{common-suffix-p-tau-z} of Lemma~\ref{lem:various-properties} and the fact that $\z \varphi(\z)$ is subscript-increasing by Part~\ref{stuff about s} of Lemma~\ref{lem:various-properties}.
\end{proof}

Corollary~\ref{cor:5over4} gives a recurrence for letters $w(6i+1)$ for sufficiently large $i$.
Letters in the other residue classes modulo $6$ are given by the next proposition, which follows directly from Theorem~\ref{thm:main}, the definition of $\varphi$, and the fact that $w(6 \cdot 1127 +0) = 0$.

\begin{proposition}\label{pro: background}
For all $i\ge 1127= \frac{\size{\p}-2}{6}$, the letters of the word $\word_{5/4}$ satisfy
\begin{align*}
	w(6 i + 0) &=
	\begin{cases}
		0		& \text{if $i \equiv 0, 3 \mod 4$} \\
		1		& \text{if $i \equiv 1, 2 \mod 4$}
	\end{cases} \\
	w(6 i + 2) &=
	\begin{cases}
		0		& \text{if $i \equiv 0, 1 \mod 4$} \\
		1		& \text{if $i \equiv 2, 3 \mod 4$}
	\end{cases} \\
	w(6 i + 3) &= 1 \\
	w(6 i + 4) &=
	\begin{cases}
		0		& \text{if $i \equiv 1, 2 \mod 4$} \\
		1		& \text{if $i \equiv 0, 3 \mod 4$}
	\end{cases} \\
	w(6 i + 5) &= 0.
\end{align*}
\end{proposition}

Recall the definition of a $k$-regular sequence~\cite{Allouche--Shallit-1992}.

\begin{definition}
Let $k \ge 2$ be an integer.
For any sequence $s(i)_{i\ge 0}$, the set of subsequences $\{s(k^e i+ j )_{i\ge 0} \colon \text{$e \ge 0$ and $0\le j \le k^e-1$} \}$ is called the \emph{$k$-kernel} of $s(i)_{i\ge 0}$.
A sequence $s(i)_{i\ge 0}$ is \emph{$k$-regular} if the $\Q$-vector space generated by its $k$-kernel is finitely generated.
The \emph{rank} of $s(i)_{i\ge 0}$ is the dimension of this vector space.
\end{definition}

The following theorem is a generalization of~\cite[Theorem 8]{Pudwell--Rowland}, which is the special case $s=0$ and $\ell=1$.

\begin{regularitytheorem*}
Let $k\ge 2$ and $\ell\ge 1$.
Let $d(i)_{i \geq 0}$ and $u(i)_{i \geq 0}$ be periodic integer sequences with period lengths $\ell$ and $k \ell$, respectively.
Let $r,s$ be nonnegative integers such that $r - s + k- 1 \ge 0$.
Let $w(i)_{i\ge 0}$ be an integer sequence such that, for all $0\le m \le k-1$ and all $i\ge 0$,
\begin{equation}\label{eqn:recurrence-regularity}
w(ki+r+m)
=
\begin{cases}
	u(ki+m) & \text{if $0 \leq m \leq k-2$} \\
	w(i+s) + d(i) & \text{if $m = k-1$}.
\end{cases}
\end{equation}
Then $w(i)_{i\ge 0}$ is $k$-regular.
\end{regularitytheorem*}

In proving Theorem~\ref{thm:regularity}, we obtain an upper bound on the rank of $w(i)_{i\ge 0}$ as a $k$-regular sequence.

Two integers $k,\ell \ge 2$ are said to be \emph{multiplicatively dependent} if there exist positive integers $\alpha$ and $\beta$ such that $k^\alpha=\ell^\beta$.
They are \emph{multiplicatively independent} if such integers do not exist.
Bell~\cite{Bell} showed that if a sequence $w(i)_{i \geq 0}$ is both $k$-regular and $\ell$-regular for multiplicatively independent integers $k$ and $\ell$ then $w(i)_{i \geq 0}$ is an eventual quasi-polynomial sequence.
Since the $i$th letter in $\word_{a/b}$ satisfies $w(i) = O(\sqrt{i})$~\cite[Theorem~9]{Pudwell--Rowland}, this implies that the value of $k$ for which the sequence of letters in $\word_{a/b}$ is $k$-regular is unique up to multiplicative dependence.

The proof of Theorem~\ref{thm:regularity} relies on the following technical lemma, which identifies kernel sequences on which Recurrence~\eqref{eqn:recurrence-regularity} can be iterated.

\begin{lemma}\label{lem:j-je-mod-powers-k}
Assume the hypotheses of Theorem~\ref{thm:regularity}.
For every $e\ge 0$, let
\[
	j_e^\star = \frac{k^e-1}{k-1} (r-s+k-1) + s.
\]
For all sufficiently large $e$, if $j \equiv j_e^\star \mod k^h$ where $0\le h \le e$, then we can iteratively apply the case $m=k-1$ of Recurrence~\eqref{eqn:recurrence-regularity} at least $h$ times to $w(k^e i + j)$ for all $i\ge \ceil{\frac{r-s}{k-1}}+1$.
\end{lemma}

\begin{proof}
The proof consists of two steps.
First, we will show that the statement holds for all sufficiently large $i$.
Second, we establish the lower bound on $i$, which is independent of $e$ for sufficiently large $e$.
We would like to iterate the case $m=k-1$ of Recurrence~\eqref{eqn:recurrence-regularity}, namely
\begin{equation}\label{eqn:case k-1}
	w(ki+r+ k - 1) = w(i+s) + d(i).
\end{equation}

Define
\begin{align*}
	f_{e,t,j}(i)
	&= k^{e-t} i + \frac{j-s}{k^t} -\frac{(1+k+\cdots+k^{t-1})(r - s + k -1)}{k^t} \\
	&=k^{e-t} i + \frac{j-s}{k^t} -\frac{k^t-1}{k^t(k-1)} (r - s + k -1),
\end{align*}
which will be useful to describe iterations of Recurrence~\eqref{eqn:case k-1}.
One checks that $f_{e,t,j}(i) + s = k f_{e,t+1,j}(i) + r +k - 1$.
Observe that if $f_{e,t+1,j}(i) \ge 0$, then $f_{e,t,j}(i)\ge 0$ since $r - s + k- 1 \ge 0$ by assumption.
Recurrence~\eqref{eqn:case k-1} gives
\begin{align}
	w(f_{e,t,j}(i)+s)
	&=w(k f_{e,t+1,j}(i) + r +k - 1) \nonumber \\
	&=w(f_{e,t+1,j}(i)+s) + d(f_{e,t+1,j}(i)) \label{eqn:special case with f}
\end{align}
if $f_{e,t+1,j}(i) \geq 0$.
We proceed by induction on $h$.
If $h = 0$, then the statement of the lemma is clear.
Let $j \equiv j_e^\star \mod k^{h + 1}$ with $0\le h\le e-1$, and inductively assume that we can iteratively apply Recurrence~\eqref{eqn:case k-1} $h$ times to $w(k^e i + j)$ for all sufficiently large $i$.
By the induction hypothesis, iteratively applying Recurrence~\eqref{eqn:special case with f} to $w(k^e i + j)$ for $t=0$, \dots, $t=h-1$ yields
\begin{equation}\label{eqn:h-applications}
	w(k^e i + j)
	= w(f_{e,0,j}(i)+s)
	= w(f_{e,h,j}(i)+s) + \sum_{t=0}^{h-1} d(f_{e,t+1,j}(i))
\end{equation}
for all sufficiently large $i$.
We claim we can apply Recurrence~\eqref{eqn:case k-1} (in the form of Recurrence~\eqref{eqn:special case with f}) one more time to obtain that $w(k^e i + j)$ is equal to
\begin{align*}
	w(& f_{e,h+1,j}(i)+s) + \sum_{t=0}^{h} d(f_{e,t+1,j}(i))\\
	&= w\!\left( k^{e-h-1} i + \frac{j-s}{k^{h+1}} -\frac{(1+k+\cdots+k^h)(r - s + k -1)}{k^{h+1}} \right)
	+ \sum_{t=0}^{h} d(f_{e,t+1,j}(i))
\end{align*}
for all sufficiently large $i$.
Indeed,
\[
\frac{j-s}{k^{h+1}} -\frac{(1+k+\cdots+k^h)(r - s + k -1)}{k^{h+1}}
\]
is an integer since, by assumption,
\begin{align*}
	j
	&\equiv j_e^\star \mod k^{h + 1} \\
	&\equiv \frac{k^e - 1}{k - 1} (r-s+k-1) + s \mod k^{h + 1} \\
	&\equiv \left(1 + k + \dots + k^{e - 2} + k^{e-1}\right) (r-s+k-1) + s \mod k^{h + 1} \\
	&\equiv \left(1 + k + \dots + k^{h - 1} + k^h\right) (r-s+k-1) + s \mod k^{h + 1}.
\end{align*}

It remains to establish the bound $i\ge \ceil{\frac{r-s}{k-1}}+1$.
To apply Recurrence~\eqref{eqn:case k-1} $h$ times to $w(k^e i + j)$, we need $f_{e,t,j}(i) \ge 0$ for all $t\in\{1,\dots,h\}$.
It follows by definition of $f_{e,t,j}(i)$ that $f_{e,t,j}(i) \ge f_{e,t,0}(i)$.
If follows from $f_{e,t,0}(i) + s = k f_{e,t+1,0}(i) + r +k - 1$ that $f_{e,t,j}(i) \ge f_{e,e,0}(i)$ since $t\le h \le e$.
We have
\[
	f_{e,e,0}(i) \ge 0 \; \Longleftrightarrow \; i \ge \frac{k^e-1}{k^e(k-1)} (r - s + k -1) + \frac{s}{k^e}.
\]
As $e$ gets large, the right side of the previous inequality approaches the finite limit $\frac{r-s+k-1}{k-1}$.
Therefore, for all sufficiently large $e$ and for all integers $i \ge \frac{r-s+k-1}{k-1}$, we can apply Recurrence~\eqref{eqn:case k-1} $h$ times to $w(k^e i + j)$, as desired.
\end{proof}

We now prove Theorem~\ref{thm:regularity}.
We will define a sequence $(j_e)_{e \geq 0}$, where each $j_e$ is the unique residue modulo $k^e$ for which we can apply Recurrence~\eqref{eqn:recurrence-regularity} the maximal number of times to $w(k^e i + j_e)$.
For $\word_{5/4}$, the sequence $(j_e)_{e \geq 0}$ is
\[
	0, 1, 31, 31, 895, 7375, 38479, 38479, 318415, 1998031, \dots.
\]
We will see in the proof of Theorem~\ref{thm: best bound on rank for w54} that applying the recurrence to $w(k^e i + j_e)$ for large $e$ produces $w(k^{e - 1} i + j_{e - 1})$.

\begin{proof}[Proof of Theorem~\ref{thm:regularity}]
We show that the $\Q$-vector space generated by the $k$-kernel of $w(i)_{i\ge 0}$ is finitely generated.
We will see that subsequences $w(k^e i + j)_{i\ge 0}$ for certain values of $j$ behave differently than others.
Namely, for most sequences, iteratively applying Recurrence~\eqref{eqn:recurrence-regularity} to all but finitely many terms brings us into the periodic background $u$ (this is Case~1 below), but for certain sequences we stay in the self-similar column (Case~2).

For every $e\ge 0$, we let
\[
	j_e = j_e^\star - k^e q_e = \frac{k^e-1}{k-1} (r-s+k-1) + s -k^e q_e,
\]
where $j_e^\star$ is defined as in Lemma~\ref{lem:j-je-mod-powers-k} and $q_e$ is the unique integer such that $0\le j_e^\star - k^e q_e < k^e$.
Since $r - s + k- 1\ge 0$, $q_e$ is nonnegative.
Since $0\le \frac{j_e}{k^e} < 1$, we have $Q_e-1 < q_e \le Q_e$ with
\[
Q_e = \frac{k^e-1}{k^e(k-1)} (r-s+k-1) + \frac{s}{k^e}.
\]
In particular, $q_e=\floor{Q_e}$.
As $e$ gets large, $Q_e$ approaches the finite limit $\frac{r-s+k-1}{k-1}$, so the integers $q_e$ are the same for all $e \ge E$ for some integer $E\ge 0$.
We take $E$ to be minimal.
We show that there exists $M \geq 0$ such that, for all $e\ge 0$ and $0\le j \le k^e-1$, $w(k^e i + j)_{i\ge 0}$ belongs to the $\Q$-vector space generated by the finite set
\[
	\left\{ w(k^e i + j)_{i \geq 0} \colon 0 \leq \text{$e \leq E - 1$ and $0\le j \le k^e-1$} \right\}
	\cup
	\left\{ w(k^e i + j_e)_{i \geq 0} \colon 0 \leq e \leq M \right\}
\]
and finitely many eventually periodic sequences.

As in the proof of Lemma~\ref{lem:j-je-mod-powers-k}, define
\[
f_{e,t,j}(i)
=k^{e-t} i + \frac{j-s}{k^t} -\frac{k^t-1}{k^t(k-1)} (r - s + k -1).
\]

\textbf{Case 1.}
First, we consider subsequences $w(k^e i + j)_{i\ge 0}$ with $0\le j \leq k^e - 1$ and $j \neq j_e$.
We show that all but finitely many have tails that can be expressed in terms of $u$ by iterating Recurrence~\eqref{eqn:recurrence-regularity}.
Let $h$ be maximal such that $j \equiv j_e \mod k^h$.
By Lemma~\ref{lem:j-je-mod-powers-k}, we can iteratively apply Recurrence~\eqref{eqn:case k-1} $h$ times to $w(k^e i + j)$ for all $e\ge E$ and $i\ge q_E + 1$, giving
\[
w(k^e i + j)
= w(f_{e,h,j}(i)+s) + \sum_{t=0}^{h-1} d(f_{e,t+1,j}(i))
\]
from Equation~\eqref{eqn:h-applications}.
Since $h$ is maximal, we cannot apply Recurrence~\eqref{eqn:case k-1} an additional time.
If $i\ge q_E+1$, then
\[
f_{e,h,j}(i) + s \ge k f_{e,e,0}(i) + r +k - 1 \ge r,
\]
as in the proof of Lemma~\ref{lem:j-je-mod-powers-k}.
So we can apply the case $m\neq k-1$ of Recurrence~\eqref{eqn:recurrence-regularity} instead to $w(f_{e,h,j}(i)+s)$.
Therefore
\begin{equation}\label{eqn: one more time}
	w(k^e i + j)
	= u(f_{e,h,j}(i)+s-r) + \sum_{t=0}^{h-1} d(f_{e,t+1,j}(i)),
\end{equation}
so $w(k^e i + j)_{i\ge q_E + 1}$ is a periodic sequence with period length at most $k \ell$ since $u$ and $d$ are periodic sequences with period lengths dividing $k \ell$.
Therefore $w(k^e i + j)_{i\ge 0}$ is an eventually periodic sequence with preperiod length at most $q_E + 1$, which is independent of $e$ and $j$.
It suffices to include generators for eventually periodic sequences with preperiod length $q_E + 1$.
Let $G_{k \ell}$ be the standard basis for periodic sequences with period length $k\ell$ (that is, with periods of the form $0, \dots, 0, 1, 0, \dots, 0$).
For all $m\ge 0$, let $v_m(i)_{i\ge 0}$ be the sequence defined by $v_m(m)=1$ and $v_m(i)=0$ for all $i\neq m$.
Let $H_{q_E + 1} = \{ v_m(i)_{i\ge 0} \colon 0\le m \le q_E\}$.
Each sequence $w(k^e i + j)_{i \geq 0}$ for $e \geq 0$ and $j\neq j_e$ belongs to the $\Q$-vector space generated by
\[
	\left\{ w(k^e i + j)_{i \geq 0} \colon 0 \leq e \leq E-1 \text{ and } 0\le j \le k^e-1 \right\}
	\cup G_{k \ell} \cup H_{q_E + 1},
\]
which is finite-dimensional.

\textbf{Case 2.}
Second, we examine subsequences $w(k^e i + j_e)_{i\ge 0}$.
We show that these sequences do not depend on $u$ and that all but finitely many of them are essentially generated by one.
We defined $j_e$ in such a way that $k^e i + j_e=f_{e,0,j_e}(i)+s$.
From Equation~\eqref{eqn:h-applications}, $e$ applications of Recurrence~\eqref{eqn:recurrence-regularity} yield
\begin{align*}
	w(k^e i + j_e)
	&= w(f_{e,e,j_e}(i) + s) + \sum_{t=0}^{e-1} d(f_{e,t+1,j_e} (i)) \\
	&= w(i-q_e + s) + \sum_{t=0}^{e-1} d(f_{e,t+1,j_e} (i))
\end{align*}
for all $i\ge q_e$, after expanding $f_{e,e,j_e}(i)$.
However, we would like to get a relation of this form that holds for all $i \ge 0$.
It is possible that $e$ applications of the recurrence are too many, and in such cases we use $h$ applications and choose $h$ accordingly.
For $h\le e$, Equation~\eqref{eqn:h-applications} with $j = j_e$ gives
\begin{equation}\label{eqn: je difference 1}
w(k^e i + j_e)
=
w\!\left( f_{e,h,j_e}(i) + s \right) + \sum_{t=0}^{h-1} d(f_{e,t+1, j_e} (i)),
\end{equation}
as long as $f_{e,h,j_e}(i)\ge 0$.
We consider two cases, because if $\frac{r-s}{k-1}$ is an integer then the integers $j_e$ are the same when $e$ is sufficiently large.

Case 2.1.
Suppose $\frac{r-s}{k-1}$ is not an integer.
The sequence $(Q_e)_{e \geq 0}$ approaches its limit from either above or below; in either case $q_e = \floor{Q_e} < \frac{r-s+k-1}{k-1}$ for sufficiently large $e$.
(Note the strict inequality, since $\frac{r-s}{k-1}$ is not an integer.)
We would like $f_{e,h,j_e}(i)\geq 0$.
Since $f_{e,h,j_e}(i)$ is an increasing function of $i$, it suffices to guarantee that $f_{e,h,j_e}(0) \geq 0$.
We get
\begin{align*}
f_{e,h,j_e}(0) \geq 0 \;
&\Longleftrightarrow \; -q_e k^{e-h} + \frac{k^{e-h}-1}{k-1} (r-s+k-1) \geq 0 \\
&\Longleftrightarrow \; k^{h-e} \leq 1 - \frac{q_e \, (k-1)}{r-s+k-1} \\
&\Longleftrightarrow \; h \leq e + \log_k \left( 1 - \frac{q_e \, (k-1)}{r-s+k-1} \right).
\end{align*}
Since $q_e < \frac{r-s+k-1}{k-1}$, the argument of the $\log_k$ is in the interval $(0,1)$.
We use the largest integer $h$ satisfying the inequality, namely
\[
	h_e = e + \floor{ \log_k \left( 1 - \frac{q_e \, (k-1)}{r-s+k-1} \right) }.
\]
We conclude that, for all $e \geq E$,
\begin{equation}\label{eqn: je difference 2}
	w(k^e i + j_e)
	- \sum_{t=0}^{h_e-1} d(f_{e,t+1, j_e} (i)),
\end{equation}
is independent of $e$ (recall that $q_e=q_E$ for all $e\ge E$).
Therefore each sequence $w(k^e i + j_e)_{i \geq 0}$ for $e \geq 0$ belongs to the $\Q$-vector space generated by
\[
	\left\{ w(k^e i + j_e)_{i \geq 0} \colon 0 \leq e \leq E \right\}
	\cup
	G_{k \ell},
\]
which is finite-dimensional.
(Recall that $G_{k \ell}$ is a basis for the periodic sequences with period length $k\ell$.)

Case 2.2.
If $\frac{r-s}{k-1}$ is an integer, then, for sufficiently large $e$, we have $q_e = \floor{Q_e} = \frac{r-s+k-1}{k-1}$.
(In this case the argument of $\log_k$ in the definition of $h_e$ would be $0$ as defined in the previous case, so we need a different approach.)
Let $E'$ be the smallest such integer $e$ such that $q_e = \frac{r-s+k-1}{k-1}$ and $f_{e,0,j_e}(i)\ge 0$ for all $i\ge 1$.
For all $e \geq E'$,
\[
	j_e = \frac{k^e - 1}{k-1} (r-s+k-1) + s -k^e q_e
	= s - \frac{r-s+k-1}{k-1}
\]
is independent of $e$; let $J = j_e$.
Since
\[
f_{e,0,j_e}(i)
=f_{e,0,J}(i)
= k^e i - \frac{r-s+k-1}{k-1},
\]
a simple computation shows that $f_{e+1,0,J}(i) + s=k f_{e,0,J}(i) + r +k-1$.
If $i = 0$, $w(k^{e+1} i +J) = w(J) = w(k^e i + J)$, so the two sequences $w(k^e i +J)_{i\ge 0}$ and $w(k^{E'} i +J)_{i\ge 0}$ agree on the first term for all $e\ge E'$.
If $i \geq 1$, then $f_{e,0,J}(i) \geq 0$ by definition of $E'$, so we can apply Recurrence~\eqref{eqn:recurrence-regularity} to obtain
\begin{align*}
	w(k^{e + 1} i + J)
	&= w(f_{e+1,0,J}(i) + s) \\
	&= w\!\left( k f_{e,0,J}(i) + r +k-1 \right) \\
	&= w\!\left(f_{e,0,J}(i) + s\right) + d(f_{e,0,J}(i)) \\
	&= w(k^e i + J) + d(f_{e,0,J}(i)).
\end{align*}
Therefore each sequence $w(k^e i + j_e)_{i \geq 0}$ for $e \geq 0$ belongs to the $\Q$-vector space generated by
\[
	\left\{ w(k^e i + j_e)_{i \geq 0} \colon 0 \leq e \leq E' \right\}
	\cup
	\left\{ \sigma(g(i)_{i\ge 1}) \colon g\in G_{k \ell} \right\},
\]
where $\sigma$ is the right shift operator, which prepends a $0$ to the front of a sequence.
Again this vector space is finite-dimensional.

We have shown that the $k$-kernel is contained in the $\Q$-vector space generated by
\begin{multline*}
	\left\{ w(k^e i + j)_{i \geq 0} \colon 0 \leq e \leq E-1 \text{ and } 0\le j \le k^e-1 \right\} \\
	\cup G_{k \ell} \cup H_{q_E + 1} \cup \left\{ w(k^e i + j_e)_{i \geq 0} \colon 0 \leq e \leq E \right\}
\end{multline*}
if $\frac{r-s}{k-1}$ is not an integer and
\begin{multline*}
	\left\{ w(k^e i + j)_{i \geq 0} \colon 0 \leq e \leq E-1 \text{ and } 0\le j \le k^e-1 \right\} \\
	\cup G_{k \ell} \cup H_{q_E + 1} \cup \left\{ w(k^e i + j_e)_{i \geq 0} \colon 0 \leq e \leq E' \right\}
	\cup \{ \sigma(g(i)_{i\ge 1}) \colon g\in G_{k \ell} \}
\end{multline*}
if $\frac{r - s}{k - 1}$ is an integer.
Then
\[
	M =
	\begin{cases}
		E	& \text{if $\frac{r - s}{k - 1}$ is not an integer} \\
		E'	& \text{if $\frac{r - s}{k - 1}$ is an integer}
	\end{cases}
\]
is the constant mentioned at the beginning of the proof.
\end{proof}

\begin{corollary}\label{cor: bound on rank for w54}
The sequence of letters in $\word_{5/4}$ is a $6$-regular sequence with rank at most $79472$.
\end{corollary}

\begin{proof}
Use $k=6$, $\ell=8$, $r=123056$, and $s=5920$ in Theorem~\ref{thm:regularity}.
Let $u = \tau(\varphi(0_0 0_1 0_2 0_3 0_4 0_5 0_6 0_7))^\omega$ be the word made up of the background, and let $d=(32321212)^\omega$.
We find $\frac{r-s+k-1}{k-1} = 23428 + \frac{1}{5}$, so $q_e \to 23428$ and $E = 7$.

Since $\frac{r - s}{k - 1}$ is not an integer, the proof of Theorem~\ref{thm:regularity} shows that the $6$-kernel of $\word_{5/4}$ is a subset of the $\Q$-vector space generated by
\begin{multline}\label{eqn: large set of generators}
	\left\{ w(k^e i + j)_{i \geq 0} \colon 0 \leq e \leq E-1 \text{ and } 0\le j \le k^e-1 \right\} \\
	\cup G_{k \ell} \cup H_{q_E + 1} \cup \left\{ w(k^e i + j_e)_{i \geq 0} \colon 0 \leq e \leq E \right\},
\end{multline}
which has dimension at most
\[
\sum_{e=0}^{E-1} k^e + k\ell + (q_E + 1) + (E + 1) = 79472.
\qedhere
\]
\end{proof}

In fact the rank is much smaller.

\begin{theorem}\label{thm: best bound on rank for w54}
The sequence of letters in $\word_{5/4}$ is a $6$-regular sequence with rank $188$.
\end{theorem}

To prove Theorem~\ref{thm: best bound on rank for w54}, we first reduce the bound from Corollary~\ref{cor: bound on rank for w54} to $4078$.

\begin{proposition}\label{pro: better bound on rank for w54}
The sequence of letters in $\word_{5/4}$ is a $6$-regular sequence with rank at most $4078$.
\end{proposition}

\begin{proof}
We use the value of the constants $k$, $\ell$, $r$, $s$, $u$, $d$, $E$, and $q_E$ from the previous proof.
Recall from the proof of Corollary~\ref{cor: bound on rank for w54} that the $k$-kernel of $\word_{5/4}$ is a subset of the $\Q$-vector space generated by~\eqref{eqn: large set of generators}.

First, we show that we can omit the generators
\[
\left\{ w(k^e i + j)_{i \geq 0} \colon 0 \leq e \leq E-1 \text{ and } 0\le j \le k^e-1 \right\}.
\]
Let $0 \leq e \leq E-1$, let $0\le j \leq k^e - 1$ such that $j \neq j_e$, and let $h$ be maximal such that $j \equiv j_e \mod k^h$.
As stated, Lemma~\ref{lem:j-je-mod-powers-k} applies for large $e$, but we show that we can apply the end of the proof to small $e$.
For the word $\word_{5/4}$, $(Q_e)_{e \geq 0}$ approaches its limit from below:
\[
(q_e)_{0\le e\le E-1} = 5920, 20510, 22941, 23347, 23414, 23425, 23427.
\]
In particular, $q_e \le q_E = 23428$ for all $e\ge 0$.
Therefore, for all $i \ge q_E+1$, we can see from the end of the proof of Lemma~\ref{lem:j-je-mod-powers-k} that $f_{e, e, 0}(i) \geq 0$, so Equation~\eqref{eqn:h-applications} holds, namely
\[
	w(k^e i + j)
	= w(f_{e,h,j}(i)+s) + \sum_{t=0}^{h-1} d(f_{e,t+1,j}(i)).
\]
Furthermore, from Case~1 in the proof of Theorem~\ref{thm:regularity}, $i \ge q_E+1$ also implies that Equation~\eqref{eqn: one more time} holds, namely
\[
	w(k^e i + j)
	= u(f_{e,h,j}(i)+s-r) + \sum_{t=0}^{h-1} d(f_{e,t+1,j}(i)).
\]
Thus $w(k^e i + j)_{i \geq 0}$ belongs to the $\Q$-vector space generated by $G_{k \ell} \cup H_{q_E + 1}$.

Additionally, we just need half the generators in $G_{k \ell}$ since $u$ agrees with $\tau(\varphi(0_0 0_1 0_2 0_3)^2)^\omega$ except on positions congruent to $5$ modulo $6$.
Let $G_{24}$ be the standard basis for periodic sequences with period length $24$.
Since $\ell$ divides $24$, each sequence $d(f_{e,h,j} (i))_{i\ge 0}$ belongs to the $\Q$-vector space generated by $G_{24}$.

Finally, we show that we do not need all generators in $H_{q_E + 1}$.
Recall that $H_{q_E + 1}$ was constructed with $q_E + 1$ generators since, for all $0 \leq j \leq k^e - 1$ such that $j \neq j_e$, the sequence $w(k^e i + j)_{i\ge 0}$ is eventually periodic with preperiod length at most $q_E + 1$.
For all $i \geq q_E +1$ we can apply Corollary~\ref{cor:5over4} $h$ times to $w(k^e i + j)$, where $h$ is maximal such that $j \equiv j_e \mod k^h$, followed by the case $m\neq k-1$ of Theorem~\ref{thm:regularity}.
We show that we can lower the bound on $i$ by using Proposition~\ref{pro: background} instead of the case $m\neq k-1$ of Theorem~\ref{thm:regularity}; namely, for all $i \geq 4046$, we can apply Corollary~\ref{cor:5over4} $h$ times to $w(k^e i + j)$, followed by Proposition~\ref{pro: background}.
This will imply that we can replace $H_{q_E + 1}$ with $H_{4046}$.
We consider large $e$ and small $e$ separately.

Let $e \geq E$, and let $0 \leq j \leq k^e - 1$ such that $j \neq j_e$.
Let $h$ be maximal such that $j \equiv j_e \mod k^h$.
In the proof of Lemma~\ref{lem:j-je-mod-powers-k}, we were able to apply Corollary~\ref{cor:5over4} $h$ times to $w(k^e i + j)$ by choosing $i$ so that $f_{e,e,0}(i) \geq 0$.
Let $i \geq 4046$.
Then
\[
	i
	> 4045 + \frac{1}{30}
	\ge \frac{k^{e - 1} - 1}{k^e(k-1)} (r - s + k -1) + \frac{s}{k^e} + \frac{\size{\p} - 2 - s}{k}
\]
since the right side of the previous inequality approaches $4045 + \frac{1}{30}$ from below as $e$ gets large.
By definition of $f$, this implies $f_{e, e - 1, 0}(i) + s \geq \size{\p} - 2$.
Since $\size{\p} - 2-s>0$, this also implies $f_{e, e - 1, 0}(i) \geq 0$.
Since $j \neq j_e$ in our current case, we have $h \neq e$, and therefore $f_{e, e - 1, 0}(i) \geq 0$ is sufficient to apply Corollary~\ref{cor:5over4} $h$ times.
This gives
\[
	w(k^e i + j)
	= w(f_{e,h,j}(i)+s) + \sum_{t=0}^{h-1} d(f_{e,t+1,j}(i))
\]
as in Case~1 of the proof of Theorem~\ref{thm:regularity}.
Since $h$ is maximal, we cannot apply Corollary~\ref{cor:5over4} an additional time.
Instead, we apply Proposition~\ref{pro: background} to $w(f_{e,h,j}(i)+s)$, since $f_{e, h, j}(i) + s \geq \size{\p} - 2$ (as opposed to $f_{e, h, j}(i) + s \geq r$ as in Case~1).
This gives
\[
	w(k^e i + j)
	= u(f_{e,h,j}(i)+s-r) + \sum_{t=0}^{h-1} d(f_{e,t+1,j}(i)),
\]
so $w(k^e i + j)_{i\ge 4046}$ is a periodic sequence with period length at most $k \ell$.
Therefore $w(k^e i + j)_{i\ge 0}$ is an eventually periodic sequence with preperiod length at most $4046$.
Let $H_{4046} = \{ v_m(i)_{i\ge 0} \colon 0\le m \le 4045\}$, where $v_m(i)_{i\ge 0}$ is the sequence defined by $v_m(m)=1$ and $v_m(i)=0$ for all $i\neq m$.

For $0 \leq e \leq E - 1$, it is sufficient to check that $f_{e, h, j}(i) + s \geq \size{\p} - 2$, since the rest of the argument is the same as the case when $e \geq E$.
A finite check shows that this inequality holds for all $i \geq 4046$.

We have shown that the $6$-kernel of $\word_{5/4}$ is a subset of the $\Q$-vector space generated by
\[
	G_{24} \cup H_{4046} \cup \left\{ w(k^e i + j_e)_{i \geq 0} \colon 0 \leq e \leq E \right\},
\]
which has dimension at most $24 + 4046 + (E + 1) = 4078$.
\end{proof}

Finally, we prove Theorem~\ref{thm: best bound on rank for w54}.

\begin{proof}[Proof of Theorem~\ref{thm: best bound on rank for w54}]
We continue to use the constants $k = 6$, $\ell = 8$, $r = 123056$, $s = 5920$, $E= 7$, and so on.
Let $j_e$ be defined as in the proof of Theorem~\ref{thm:regularity}.
Let
\[
	V = \left\langle w(k^e i + j)_{i \geq 0} \colon 0 \leq e \leq E \text{ and } 0\le j \le k^e-1 \text{ with } j\neq j_e \right\rangle
\]
be the vector space generated by the kernel sequences $w(k^e i + j)_{i \geq 0}$ with $0 \leq e \leq E$ and $j\neq j_e$, and let
\[
	W = \left\langle w(k^e i + j_e)_{i \geq 0} \colon 0 \leq e \leq E\right\rangle.
\]
We show that $\dim V = 179$, $\dim W = 8$, and the vector space generated by the $k$-kernel of $w(i)_{i \geq 0}$ is the direct sum $V \oplus W \oplus \langle w(k^{E+1} i + j_E)_{i \geq 0}\rangle$, which has dimension $188$.

Let $e\ge 0$ and $0\le j \le k^e-1$ such that $j\neq j_e$.
We claim that the first $4050$ terms of the sequence $w(k^e i + j)_{i\ge 0}$ determine it uniquely.
From the proof of Proposition~\ref{pro: better bound on rank for w54}, $w(k^e i + j)_{i\ge 4046}$ is periodic.
Since we used Proposition~\ref{pro: background} to obtain the bound $4046$, the period length of $w(k^e i + j)_{i\ge 4046}$ is a divisor of $4$.
Therefore the first $4046+4$ terms determine it uniquely.

In particular, the first $4050$ terms of each generator of $V$ determine the sequence uniquely.
Therefore we obtain the dimension of $V$ by row-reducing the matrix containing the first $4050$ terms of each sequence $w(k^e i + j)_{i\ge 0}$ with $0 \leq e \leq E$, $0\le j \le k^e-1$, and $j\neq j_e$.
This gives dimension $179$ and took about a half hour using $8$ parallel threads.

In computing the dimension of $V$, we computed a basis of $V$ consisting of kernel sequences.
We claim that all periodic sequences with period length $4$ and the sequence $1,0,0,\ldots$ belong to $V$.
For $m\in\{0,1,2,3\}$, define $g_m(i) = 1$ if $i \equiv m \mod 4$ and $g_m(i) = 0$.
Let $G_4=\{g_m(i)_{i \geq 0}\colon 0\le m\le 3\}$, and let $H_1 = \{v_0(i)_{i \geq 0}\}$, where $v_0(i)_{i \geq 0}$ is the eventually $0$ sequence as defined in the proof of Proposition~\ref{pro: better bound on rank for w54}.
Each sequence in $G_4 \cup H_1$ is eventually periodic with preperiod length $\leq 4046$ (in fact $\leq 1$) and period length dividing $4$, so row-reducing a $184$-row matrix using the first $4050$ terms shows that $G_4\cup H_1 \subset V$.
This computation took less than a second and finds the relations
\begin{align*}
	-w(6 i) + w(36 i) + w(6 i + 2) - 2 w(6 i + 4) + 2 g_0(i) &= 0 \\
	-w(6 i) - w(36 i) + w(6 i + 2) + 2 g_1(i) &= 0 \\
	-w(6 i) + w(36 i) - w(6 i + 2) + 2 g_2(i) &= 0 \\
	w(6 i) - w(36 i) - w(6 i + 2) + 2 g_3(i) &= 0 \\
	-w(216 i + 23) - w(216 i + 29) + w(216 i + 35) + w(216 i + 41) + v_0(i) &= 0.
\end{align*}
In particular, since $G_4 \subset V$, the constant sequence $(1)_{i \geq 0}$ is an element of $V$.

We show that $w(k^e i + j_e)_{i \geq 0} \notin V$ for $0 \leq e \leq E$, that $w(k^{E+1} i + j_E)_{i \geq 0} \notin V$, and that all $188$ sequences are linearly independent of each other, by using the first $4050$ terms of the $179$ basis elements of $V$ and row-reducing the appropriate $188$-row matrix.
This computation took less than a second.
In particular, $\dim W = E + 1 = 8$.

By definition, the kernel sequences $w(k^e i + j)_{i \geq 0}$ for all $0\le e \le E$ belong to $V \oplus W $.
We have proved that this vector space has dimension $187$.
It remains to show that for $e \geq E + 1$ all the kernel sequences $w(k^e i + j)_{i \geq 0}$ belong to $V \oplus W \oplus \langle w(k^{E+1} i + j_E)_{i \geq 0}\rangle$.
First we consider $j = j_e$.
We show that
\[
w(k^e i + j_e) = w(k^E i + j_E) + 2 (e - E).
\]
for all $i\ge 0$.
From~\eqref{eqn: je difference 1} and~\eqref{eqn: je difference 2}, we know that
\begin{equation}\label{eqn: je difference 3}
w(k^e i + j_e)
=
w\!\left( f_{e,h_e,j_e}(i) + s \right)
+ \sum_{t=0}^{h_e-1} d(f_{e,t+1, j_e} (i)),
\end{equation}
for all $i\ge 0$.
We next compute $h_e$ and $f_{e,t+1, j_e} (i) \bmod{\ell}$ for $t\in\{0,1,\ldots,h_e-1\}$.
Since $q_e=q_E=23428$ for all $e\ge E$, we have
\[
h_e = e + \floor{ \log_k \left( 1 - \frac{q_E \, (k-1)}{r-s+k-1} \right) } = e - E.
\]
By definition, we have
\begin{align*}
f_{e,t+1, j_e} (i)
&= k^{e-t-1}i+ \frac{k^{e-t-1}-1}{k-1} (r-s+k-1) -q_e k^{e-t-1}.
\end{align*}
In Equation~\eqref{eqn: je difference 3}, $t+1\le h_e = e-E$, so $e-t-1 \ge E=7$.
Therefore $k^{e-t-1}\equiv 0 \mod{\ell}$, so we get
\[
f_{e,t+1, j_e} (i)
\equiv \frac{-1}{k-1} (r-s+k-1)
\equiv 7 \mod{\ell}
\]
for all $i\ge 0$.
Since $d(7)=2$, the sum on the right side of \eqref{eqn: je difference 3} is
\[
\sum_{t=0}^{h_e-1} d(f_{e,t+1, j_e} (i))
= \sum_{t=0}^{h_e-1} d(7)
= 2 (e-E).
\]
Thus Equation~\eqref{eqn: je difference 3} becomes
\begin{align*}
w(k^e i + j_e)
&= w\!\left( f_{e,h_e,j_e}(i) + s \right) + 2 (e-E) \\
&= w\!\left(k^{e-h_e}i + \frac{k^{e-h_e}-1}{k-1} (r-s+k-1) -q_E k^{e-h_e} + s \right) + 2 (e-E) \\
&= w\!\left(k^E i + \frac{k^E-1}{k-1} (r-s+k-1) + s -q_E k^{E} \right) + 2 (e-E) \\
&= w\!\left(k^E i + j_E \right) + 2 (e-E)
\end{align*}
as desired.
Therefore $w(k^ei+j_e)_{i\ge 0}$ is a linear combination of the kernel sequence $w(k^Ei+j_E)_{i\ge 0} \in W$ and the constant sequence $(1)_{i \geq 0} \in V$.

Now let $0 \leq j \leq k^e - 1$ such that $j \neq j_e$.
We use induction on $e$, so assume that $w(k^{e - 1} i + j')_{i \geq 0} \in V \oplus W \oplus \langle w(k^{E+1} i + j_E)_{i \geq 0}\rangle$ for all $0 \leq j' \leq k^{e - 1} - 1$.
If $j \nequiv 1 \mod k$, then Proposition~\ref{pro: background} implies that $w(k^e i + j)_{i\ge 1}$ is periodic with period length dividing $4$.
Therefore $w(k^e i + j)_{i\ge 0}$ is a linear combination of sequences in $G_4 \cup H_1 \subset V$.
Now assume $j \equiv 1 \mod k$.
From Equations~\eqref{eqn:special case with f} and~\eqref{eqn:h-applications}, we have
\[
w(k^ei+j)
= w(f_{e,0,j}(i)+s)
= w(f_{e,1,j}(i)+s) + d(f_{e,1,j}(i)).
\]
We have
\[
f_{e,1,j}(i) =k^{e-1} i + \frac{j - s}{k} - \frac{r - s + k -1}{k}
	= k^{e-1} i + \frac{j- r - k +1}{k}.
\]
Note that $j- r - k +1\equiv 0 \mod{k}$, so $f_{e,1,j}(i)$ is an integer.
Note also that $f_{e, 1, j}(i) \equiv \frac{j- r - k +1}{k} \mod \ell$, so $d(f_{e,1,j}(i))_{i \geq 0}$ is a constant sequence.
We consider two subcases depending on whether $j<r-ks+k-1=87541$ or not.

If $r-ks+k-1 \le j \le k^e-1$, then $\frac{j- r - k +1}{k} + s \in \{0,1,\ldots,k^{e-1}-1\}$, so the sequence $w(f_{e,1,j}(i)+s)_{i\ge 0}$ that appears in
\[
	w(k^ei+j)
	= w(f_{e,1,j}(i)+s) + d(f_{e,1,j}(i))
\]
is the kernel sequence $w(k^{e-1} i + \frac{j- r - k +1}{k}+s)_{i\ge 0}$.
Therefore $w(k^ei+j) \in V \oplus W \oplus \langle w(k^{E+1} i + j_E)_{i \geq 0} \rangle$ by the induction hypothesis.

If $0 \le j < r-ks+k-1$, then $\frac{j- r - k +1}{k} + s < 0 $, so the sequence $w(k^{e-1} i + \tfrac{j- r - k +1}{k} + s)_{i \geq 0}$ that appears in
\[
	w(k^ei+j)
	= w(k^{e-1} i + \tfrac{j- r - k +1}{k} + s) + d(f_{e,1,j}(i)).
\]
is not necessarily a kernel sequence.
We checked above that $w(k^{E+1} i + j_E)_{i \geq 0}$ is linearly independent of the sequences in $V \oplus W$.
For all other $j$, we show that $w(k^ei+j)_{i \geq 0} \in V \oplus W$.
(We do this directly without using the inductive hypothesis.)

For each $e\in\{8,9,\ldots,12\}$ and each $0 \leq j \leq k^e - 1$ such that $(e,j)\neq (E+1,j_E)$, we check that $w(k^ei+j)_{i \geq 0} \in V \oplus W \oplus \langle w(k^{E+1} i + j_E)_{i \geq 0}\rangle$ by row-reducing a $189$-row matrix using the first $4050$ terms.
This computation took about $40$ minutes using $8$ parallel threads.

Now assume $e\ge 13$.
We show that each $w(k^ei+j)_{i \geq 1}$ is a constant sequence.
Since $j_{E+1}=318415 > r-ks+k-1 > j$, $j-j_e$ is not divisible by $k^{E+1}$.
Therefore $h\le 7$.
As in Proposition~\ref{pro: better bound on rank for w54}, to apply Corollary~\ref{cor:5over4} $h$ times and Proposition~\ref{pro: background}, we need $f_{e, t, j}(i)\ge 0$ for all $t\in\{1,\ldots,h\}$ and $f_{e, h, j}(i) + s \geq \size{\p} - 2$.
We show that this happens for all $i\ge 1$.
Since $\size{\p} - 2 -s >0$, it is enough to check $f_{e, h, j}(i) \geq \size{\p} - 2 -s$ for all $i\ge 1$.
We have
\begin{align*}
i> \tfrac{6794085199}{13060694016} &= \frac{s}{k^{13}} + \frac{k^7-1}{k^{13}(k-1)} (r - s + k -1) + \frac{k^7}{k^{13}} (\size{\p} - 2 - s) \\
&\ge \frac{s-j}{k^e} + \frac{k^h-1}{k^e(k-1)} (r - s + k -1) + \frac{k^h}{k^e} (\size{\p} - 2 - s),
\end{align*}
which implies
\[
f_{e,h,j}(i)
= k^{e-h} i + \frac{j-s}{k^h} -\frac{k^h-1}{k^h(k-1)} (r - s + k -1)
\ge \size{\p} - 2 - s.
\]
So
\[
w(k^e i + j)
= w(f_{e,h,j}(i)+s) + \sum_{t=0}^{h-1} d(f_{e,t+1,j}(i))
\]
holds for all $i\ge 1$.
We show that $f_{e,t+1,j}(i) \bmod 24$ is independent of $i$ for all $t\in\{0,1,\ldots,h-1\}$; this will imply that $w(f_{e,h,j}(i)+s)$ is constant by Proposition~\ref{pro: background} and that $d(f_{e,t+1,j}(i))$ is constant since $\ell$ divides $24$.
We have
\begin{align*}
f_{e,t+1,j}(i)
&= k^{e-t-1} i + \frac{j-s}{k^{t+1}} -\frac{k^{t+1}-1}{k^h(k-1)} (r - s + k -1) \\
&\equiv \frac{j-s}{k^{t+1}} -\frac{k^{t+1}-1}{k^h(k-1)} (r - s + k -1) \mod{24}
\end{align*}
since $e-t-1\ge 6$.
Therefore $w(k^ei+j)_{i \geq 1}$ is a constant sequence, so $w(k^ei+j)_{i \geq 0}$ is a linear combination of sequences in $G_4 \cup H_1 \subset V$.
\end{proof}

\section{Open questions}\label{sec:open-questions}

We end this paper with several open questions.
Regarding finite alphabets, the structure of the lexicographically least square-free infinite word on $\{0,1,2\}$ is still unknown~\cite[Open Problem~2 in Section~1.10]{Allouche--Shallit-2003}.

The following table presents the known information about $\word_{a/b}$ for simple rational numbers $a/b$.
\[
\begin{array}{cccccccr}
a/b & k & d & r' & s & \text{rank} & \text{result} & \text{sequence} \\
\hline
a \in \N_{\ge 2} & a & 1 & 0 & 0 & 2 & \text{\cite{Guay--Shallit}} & \text{e.g.\ \seq{A007814}} \\
3/2 & 6 & 2 & 0 & 0 & 3 & \text{\cite{Rowland--Shallit}} & \text{\seq{A269518}} \\
4/3 & 56 & 1,2 & 73 & 0 & 4 & \text{\cite[Theorem~7]{Pudwell--Rowland}} & \text{\seq{A277142}} \\
5/3 & 7 & 1 & 0 & 0 &  2 & \text{\cite[Theorem~1]{Pudwell--Rowland}} & \text{\seq{A277143}} \\
5/4 & 6 & 1,2,3 & 123061 & 5920 & 188 & \text{Theorem~\ref{thm: best bound on rank for w54}} & \text{\seq{A277144}} \\
7/4 & 50847 & 2 & 0 & 0 & 2 & \text{\cite[Theorem~4]{Pudwell--Rowland}} & \text{\seq{A277145}} \\
6/5 & 1001 & 3 & 30949 & 0 & 33 & \text{\cite[Theorem~5]{Pudwell--Rowland}} & \text{\seq{A277146}} \\
7/5 &  &  &  &  &  & \text{\cite[Conjecture~6]{Pudwell--Rowland}} & \text{\seq{A277147}} \\
8/5 & 733 & 2 & 0 & 0 & 2 & \text{\cite[Theorem~3]{Pudwell--Rowland}} & \text{\seq{A277148}} \\
9/5 & 13 & 1 & 0 & 0 & 2 & \text{\cite[Theorem~2]{Pudwell--Rowland}} & \text{\seq{A277149}} \\
7/6 &  &  &  & &  & \text{Conjecture~\ref{conj: unknown words}} & \text{\seq{A277150}} \\
11/6 &  &  &  &  &  & \text{No conjecture} & \text{\seq{A277151}} \\
8/7 &  &  &  &  &  & \text{Conjecture~\ref{conj: unknown words}} & \text{\seq{A277152}} \\
9/7 &  &  &  &  &  & \text{Conjecture~\ref{conj: unknown words}} & \text{\seq{A277153}} \\
10/7 &  &  &  &  &  & \text{Conjecture~\ref{conj: unknown words}} & \text{\seq{A277154}} \\
11/7 & 27 & 1 & 0 & 0 & 2 & \text{\cite[Theorem~52]{Pudwell--Rowland}} & \text{\seq{A277155}} \\
12/7 & 17  & 1 & 0 & 0 & 2 & \text{\cite[Theorem~16]{Pudwell--Rowland}} & \text{\seq{A277156}} \\
13/7 & 19  & 1 & 0 & 0 & 2 & \text{\cite[Theorem~16]{Pudwell--Rowland}} & \text{\seq{A277157}} \\
9/8 &  &  &  &  &  & \text{No conjecture} & \text{\seq{A277158}} \\
11/8 &  &  &  &  &  & \text{No conjecture} & \text{\seq{A277159}} \\
13/8 & 33 & 1 & 0 & 0 & 2 & \text{\cite[Theorem~52]{Pudwell--Rowland}} & \text{\seq{A277160}} \\
15/8 &  &  &  &  &  & \text{Conjecture~\ref{conj: unknown words}} & \text{\seq{A277161}}
\end{array}
\]
Here $k$ is the smallest value for which the sequence of letters in $\word_{a/b}$ is $k$-regular, and we include the values of $d$ which arise in the morphism.
The values of $r'$ and $s$ are as in Recurrence~\eqref{eqn:double transient general form} and are chosen to be minimal.
The rank of each sequence for which $s=0$ can be determined from the recurrence it satisfies.
This table emphasizes the extent to which $\word_{5/4}$ is more complicated than other words.

Previous work on $\word_{a/b}$ has suggested that the structure of the word $\word_{a/b}$ is generally more complicated for even denominators $b$ than for odd $b$.
This trend is supported by the structure of $\word_{5/4}$.

An obvious question is whether the proof strategy for $\word_{5/4}$ can be applied to other words $\word_{a/b}$.
It seems likely that it can, but we leave this an open question.
The major difficulties are that identifying the structure of $\word_{a/b}$ potentially requires computing a huge number of terms and that we do not have a systematic way of guessing the structure even if we have many terms.

Ordered by denominator, the simplest words whose structure is not yet known are $\word_{7/5}$, $\word_{7/6}$, and $\word_{11/6}$.
Pudwell and the first-named author~\cite[Conjecture~6]{Pudwell--Rowland} conjectured that the letters of $\word_{7/5}$ satisfy
\[
	w(80874 i + 173978) = w(i) + 1
\]
for all $i \geq 0$.
Here we conjecture the structure of five additional words.

\begin{conjecture}\label{conj: unknown words}
For all $i \geq 0$, the letters of $\word_{7/6}$ satisfy
\[
	w(41190 i + 41201) = w(i) + 3.
\]
For all $i \geq 0$, the letters of $\word_{8/7}$ satisfy
\[
	w(340 i + 52670) = w(i) + 3.
\]
For all $i \geq 0$, the letters of $\word_{9/7}$ satisfy
\[
	w(44 i + 2701 ) = w(i) + 2.
\]
For all $i \geq 0$, the letters of $\word_{10/7}$ satisfy
\[
	w(26 i + 428)
	= w(i) +
	\begin{cases}
	0 & \text{if } i=0 \\
	1 & \text{if } i\neq 0.\\
	\end{cases} 	
\]
For all $i \geq 0$, the letters of $\word_{15/8}$ satisfy
\[
	w(22763 i + 22850) = w(i) + 2.
\]
\end{conjecture}

However, we still do not have a conjecture for $\word_{11/6}$ and $\word_{11/8}$.
Based on $\word_{3/2}$ and $\word_{5/4}$, one might guess that $\word_{7/6}$ and $\word_{9/8}$ have similar structure.
In Conjecture~\ref{conj: unknown words}, the recurrence for $\word_{7/6}$ has a single value of $d$ and $s=0$, so it does not seem to be part of the same family.
For $\word_{9/8}$, we do not have a conjectural recurrence, but experiments suggest that $k=156$ is promising.

Additionally, there are other natural notions of pattern avoidance for fractional powers on $\N$.
For $a/b > 1$, we define two additional words.
Let $\word_{\geq a/b}$ be the lexicographically least infinite word on $\N$ avoiding $p/q$-powers for all $p/q \geq a/b$, and let $\word_{> a/b}$ be the lexicographically least infinite word on $\N$ avoiding $p/q$-powers for all $p/q > a/b$.

Guay-Paquet and Shallit~\cite{Guay--Shallit} asked whether $\word_{\geq 5/2}$ is in fact a word on $\{0,1,2\}$.
This question is still open.
Pudwell and Rowland~\cite[Theorem~71]{Pudwell--Rowland} proved that $\word_{27/23}$ is a word on the finite alphabet $\{0,1,2\}$, showing that there exist lexicographically least pattern-avoiding words defined on $\N$ that only use a finite alphabet.

Pudwell and Rowland~\cite[Conjecture~13]{Pudwell--Rowland} conjectured that
\[
	\word_{\geq 4/3}(336i + 1666) = \word_{4/3}(56i + 17) + 4
\]
for all $i \geq 0$.
This suggests that the structure of $\word_{\geq a/b}$ is slightly more complicated than that of $\word_{a/b}$.
However, not much is known about $\word_{\geq a/b}$.
Guay--Paquet and Shallit \cite{Guay--Shallit} showed that the overlap-free word~\citeseq{A161371}
\[
\word_{>2} = 001001100100200100110010021001002001001100\cdots
\]
is generated by a non-uniform morphism, which leads us to believe that the structure of $\word_{> a/b}$ is even more complicated than that of $\word_{a/b}$.

The biggest question remains the following.
For each $a/b$, is there an integer $k\ge 2$ such that the sequence of letters in $\word_{a/b}$ is $k$-regular?
Similarly, one could ask about $\word_{\ge a/b}$.
Finally, which words $\word_{>a/b}$ are $k$-regular for some $k$?


\begin{thebibliography}{99}

\bibitem{Allouche--Currie--Shallit}
J.-P. Allouche, J. Currie, J. Shallit, Extremal infinite overlap-free binary words, \textit{Electron. J. Combin.} \textbf{5} (1998), Research paper 27, 11 pages.

\bibitem{Allouche--Shallit-1992}
J.-P. Allouche, J. Shallit, The ring of $k$-regular sequences, \textit{Theoret. Comput. Sci.} \textbf{98} (1992) 163--197.

\bibitem{Allouche--Shallit-2003}
J.-P. Allouche, J. Shallit, \textit{Automatic Sequences: Theory, Applications, Generalizations}, Cambridge University Press, Cambridge, 2003.

\bibitem{Bell} J. P. Bell, A generalization of Cobham's theorem for regular sequences, \textit{S\'em. Lothar. Combin.} \textbf{54A} (2005/2007), Article B54Ap., 15 pages.

\bibitem{Berstel}
J. Berstel, Axel Thue's Papers on Repetitions in Words: A Translation, \textit{Publications du LaCIM} \textbf{20}, Universit\'e du Qu\'ebec \`a Montr\'eal, 1995.

\bibitem{Berstel--Perrin}
J. Berstel, D. Perrin, The origins of combinatorics on words, \textit{European J. Combin.} \textbf{28} (2007), 996--1022.

\bibitem{Dejean}
F. Dejean. Sur un th\'eor\`eme de Thue, \textit{J. Combinatorial Theory Ser. A} \textbf{13} (1972), 90--99.

\bibitem{Guay--Shallit}
M. Guay-Paquet, J. Shallit, Avoiding squares and overlaps over the natural numbers, \textit{Discrete Math.} \textbf{309} (2009), no. 21, 6245--6254.

\bibitem{Pudwell--Rowland}
L. Pudwell, E. Rowland, Avoiding fractional powers over the natural numbers, \textit{Electron. J. Combin.} \textbf{25} (2018), no. 2, Paper 2.27, 46 pages.

\bibitem{Rowland--Shallit}
E. Rowland, J. Shallit, Avoiding $3/2$-powers over the natural numbers, \textit{Discrete Math.} \textbf{312} (2012), no. 6, 1282--1288.

\bibitem{OEIS}
N. Sloane et al.,
The On-Line Encyclopedia of Integer Sequences,
\url{http://oeis.org}.

\bibitem{Thue-1906}
A. Thue, \"Uber unendliche Zeichenreihen, \textit{Norske Vid. Selsk. Skr. I Math-Nat. Kl.} \textbf{7} (1906), 1--22.

\bibitem{Thue-1912}
A. Thue, \"Uber die gegenseitige Loge gleicher Teile gewisser Zeichenreihen, \textit{Norske Vid. Selsk. Skr. I Math-Nat. Kl. Chris.} \textbf{1} (1912), 1--67.

\end{thebibliography}
\end{document}